\documentclass{amsart}
\usepackage{amsfonts}
\usepackage{bbm}
\usepackage{mathrsfs}
\usepackage{amsmath}
\usepackage{bigdelim}
\usepackage{multirow}
\usepackage{amssymb}
\usepackage{indentfirst}
\usepackage{CJK}
\usepackage{fancyhdr}
\usepackage{amscd,amssymb,amsmath,graphicx,verbatim,xy}
\usepackage[TS1,OT1,T1]{fontenc}

\newtheorem{theorem}{Theorem}[section]
\newtheorem{lemma}[theorem]{Lemma}
\newtheorem{corollary}[theorem]{Corollary}
\newtheorem{proposition}[theorem]{Proposition}
\newtheorem{conjecture}[theorem]{Conjecture}
\theoremstyle{definition}
\newtheorem{definition}[theorem]{Definition}
\newtheorem{example}[theorem]{Example}

\newtheorem{question}[theorem]{Question}
\theoremstyle{remark}
\newtheorem{remark}[theorem]{Remark}
\numberwithin{equation}{section}

\begin{document}

	\title[Analytic automorphism group and similar representation]{Analytic automorphism group and similar representation of analytic functions}

	\author{Bingzhe Hou}
	\address{Bingzhe Hou, School of Mathematics, Jilin University, 130012, Changchun, P. R. China}
	\email{houbz@jlu.edu.cn}
	
	\author{Chunlan Jiang}
	\address{Chunlan Jiang, Department of Mathematics, Hebei Normal University, 050016, Shijiazhuang, P. R. China}
	\email{cljiang@hebtu.edu.cn}
	
	\date{}
	\thanks{}
	\begin{abstract}
		In geometry group theory, one of the milestones is M. Gromov's polynomial growth theorem: Finitely generated groups have polynomial growth if and only if they are virtually nilpotent. Inspired by M. Gromov's work, we introduce the growth types of weighted Hardy spaces. In this paper, we focus on the weighted Hardy spaces of polynomial growth, which cover the classical Hardy space, weighted Bergman spaces, weighted Dirichlet spaces and much broader. Our main results are as follows. $(1)$ We obtain the boundedness of the composition operators with symbols of analytic automorphisms of unit open disk acting on weighted Hardy spaces of polynomial growth, which implies the multiplication operator $M_z$ is similar to $M_{\varphi}$ for any analytic automorphism $\varphi$ on the unit open disk. Moreover,  we obtain the boundedness of composition operators induced by analytic functions on the unit closed disk on weighted Hardy spaces of polynomial growth. $(2)$ For any Blaschke product $B$ of order $m$, $M_B$ is similar to $\bigoplus_{1}^m M_z$, which is an affirmative answer to a generalized version of a question proposed by R. Douglas in 2007.  $(3)$ We also give counterexamples to show that the composition operators with symbols of analytic automorphisms of unit open disk acting on a weighted Hardy space of intermediate growth could be unbounded, which indicates the necessity of the setting of polynomial growth condition.  Then, the collection of weighted Hardy spaces of polynomial growth is almost the largest class such that Douglas's question has an affirmative answer. $(4)$ Finally,  we give the Jordan representation theorem and similarity classification for the analytic functions on the unit closed disk as multiplication operators on a weighted Hardy space of polynomial growth. 
	\end{abstract}
	\maketitle
	\tableofcontents

\section{Introduction}

Denote by $\textrm{Hol}(\mathbb{D})$ the space of all analytic functions on the unit open disk $\mathbb{D}$, and denote by $\textrm{Aut}(\mathbb{D})$ the analytic automorphism group on $\mathbb{D}$, which is the set of all analytic bijections from $\mathbb{D}$ to itself. As well known, each $\varphi\in\textrm{Aut}(\mathbb{D})$ could be written as the following form
\[
\varphi(z)={\textrm e}^{\mathbf{i}\theta}\cdot \frac{z_0-z}{1-\overline{z_0}z}, \ \ \ \text{for some} \ z_0\in\mathbb{D}.
\]
We are interested in a subclass of analytic functions on the unit open disk $\mathbb{D}$, named weighted Hardy space.

In this paper, we introduce the weighted Hardy space from a given weight sequence.
For any $f\in \textrm{Hol}(\mathbb{D})$, denote the Taylor expansion of $f(z)$ by
\[
f(z)=\sum\limits_{k=0}^{\infty}\widehat{f}(k)z^k.
\]
Let $w=\{w_k\}_{k=1}^{\infty}$ be a sequence of positive numbers. Write $\beta=\{\beta_k\}_{k=0}^{\infty}$,
\[
\beta_0=1, \ \ \text{and} \ \ \beta_k=\prod\limits_{j=1}^{k}w_{j}, \ \ \text{for} \ k\geq1.
\]
The weighted Hardy space $H^2_{\beta}$ induced by the weight sequence $w$ (or $\beta$) is defined by
\[
H^2_{\beta}=\{f(z)=\sum\limits_{k=0}^{\infty}\widehat{f}(k)z^k; \ \sum\limits_{k=0}^{\infty}|\widehat{f}(k)|^{2}{\beta}_k^2<\infty\}.
\]
Moreover, the weighted Hardy space $H^2_{\beta}$ is a complex separable Hilbert space, on which the inner product is defined by,  for any $f,g \in H^2_{\beta}$
\[
\langle  f, g\rangle_{H^2_{\beta}}=\sum\limits_{k=0}^{\infty}{\beta}_k^2\overline{\widehat{g}(k)}\widehat{f}(k)
\]
Then, any $f\in H^2_{\beta}$ has the following norm
\[
\|f(z)\|_{H^2_{\beta}}=\sqrt{\langle  f, f\rangle_{H^2_{\beta}}}=\sqrt{\sum\limits_{k=0}^{\infty}{\beta}_k^2|\widehat{f}(k)|^{2}}.
\]
In particular, $\|z^n\|_{H^2_{\beta}}={\beta}_n$ for each $n\in \mathbb{\mathbb{N}}$. Let $H^2_{\beta}$ and $H^2_{\beta'}$ be two weighted Hardy spaces. If there are positive constants $K_1$ and $K_2$ such that $K_1\leq \frac{\beta'_k}{\beta_k}\leq K_2$ for all $k\in\mathbb{N}$, then $H^2_{\beta}=H^2_{\beta'}$ and the norms are equivalent, and hence we say that $H^2_{\beta}$ and $H^2_{\beta'}$ are equivalent weighted Hardy spaces.

Furthermore, any $f\in H^2_{\beta}$ induces a multiplication operator $M_f$, defined by
\[
M_f(g)=f\cdot g, \ \  \ \ \text{for \ any }\ g\in H^2_{\beta}.
\]
Define $H^{\infty}_{\beta}$ the set
\[
H^{\infty}_{\beta}=\{f(z)\in H^2_{\beta}; \ M_f \ \text{is  a  bounded  operator  from} \ H^2_{\beta} \ \text{to} \ H^2_{\beta}\}.
\]
In the present article, we always assume the weight sequence $w$ satisfying
\[
\lim\limits_{k\rightarrow\infty}w_k=1.
\]
In this case, we have
\[
\textrm{Hol}(\overline{\mathbb{D}})\subseteq H^{\infty}_{\beta}\subseteq H^2_{\beta} \subseteq \textrm{Hol}(\mathbb{D}),
\]
where $\textrm{Hol}(\overline{\mathbb{D}})$ is denoted by the space of all analytic functions on the unit closed disk $\overline{\mathbb{D}}$. It suffices to check the first inclusion relationship above. Let $H^{2}_{\beta}$ be a weighted Hardy space induced by the weight sequence $w=\{w_k\}_{k=1}^{\infty}$. Following from the work of Shields \cite{Shi}, to study the multiplication operator $M_z$ on $H^{2}_{\beta}$ is equivalent to study the forward unilateral weighted shift $S_{w}$ on the classical Hardy space, where $S_{w}:H^2\rightarrow H^2$ is defined by
\[
S_{w}(z^k)=w_{k+1}z^{k+1}, \ \ \ \text{for} \ k=0,1,2,\ldots.
\]
Notice that $w_k\rightarrow 1$ implies the spectrum of $S_{w}$, denoted by $\sigma(S_w)$, is the unit closed disk $\overline{\mathbb{D}}$. Then, for any $f\in \textrm{Hol}(\overline{\mathbb{D}})$, one can see that the analytic function calculus $f(S_w)$ on $H^2$ is corresponding to the multiplication operator $M_f$ on $H^{2}_{\beta}$. Consequently, we have $\sigma(M_f)=f(\overline{\mathbb{D}})$. Therefore, $M_f$ is a bounded operator on $H^{2}_{\beta}$ and moreover, $M_f$ is lower bounded on $H^{2}_{\beta}$ if and only if $f$ has no zero point in $\partial\mathbb{D}$.

In addition, if $w_k\rightarrow 1$, the weighted Hardy space $H^2_{\beta}$ is isometric to the following weighted square summable sequence space
\[
\{(c_0, c_1, c_2, \ldots); \ \sum\limits_{k=0}^{\infty}|c_k|^{2}{\beta}_k^2<\infty\}.
\]
The classical Hardy space, the weighted Bergman spaces, and the weighted Dirichlet spaces are all the weighted Hardy spaces of this type.

Then, every analytic function $f$ in $\textrm{Hol}(\overline{\mathbb{D}})$ has an operator representation $M_f$ on $H^2_{\beta}$. As what we do in the representation theory or in linear algebra essentially, a natural task is to study the similarity of the representations. More precisely, in finite-dimensional matrix theory, the Jordan decomposition theorem tells us each finite-dimensional matrix is similar to a direct sum of some Jordan blocks.
With regard to infinite-dimensional operators (or matrices), strongly irreducible operator is a suitable substitute for Jordan block (see \cite{JW}). Consequently, strongly irreducible decomposition is seemed as the Jordan decomposition of an infinite-dimensional operator. Then we could study the "Jordan decomposition" and similarity classification of the representation of $\textrm{Hol}(\overline{\mathbb{D}})$ on a weighted Hardy space $H^2_{\beta}$ with $w_k\rightarrow 1$. Notice that there is no standard form of strongly irreducible operator. It is also valuable to describe when $M_f$ is similar to $M_g$ if $M_f$ and $M_g$ are strongly irreducible operators on $H^2_{\beta}$, which is related to the boundedness of composition operator $C_{\varphi}$ for $\varphi\in\textrm{Aut}(\mathbb{D})$.

In \cite{JZ}, Jiang and Zheng studied the similarity of the operator representation $M_f$ on the weighted Bergman spaces. Furthermore, Ji and Shi studied the similarity of the operator representation $M_f$ on the Sobolev disk algebra in \cite{Ji}. More recently, E. Gallardo-Guti\'{e}rrez and J. Partington \cite{GP22} obtained some similar results of Jiang and Zheng \cite{JZ}. In addition, the boundedness of composition operator $C_{\varphi}$ for $\varphi\in\textrm{Aut}(\mathbb{D})$ has been obtained in the weighted Bergman spaces, the weighted Dirichlet spaces, Sobolev space and so on (we refer to \cite{C95}). In this article, we aim to study the similarity of the operator representation $M_f$ on more general spaces. Firstly, we will introduce the growth types of weighted Hardy spaces, which are inspired by the growth types of finitely generated groups. In geometry group theory, one of the milestones is M. Gromov's polynomial growth theorem in \cite{G81}: "Finitely generated groups have polynomial growth if and only if they are virtually nilpotent." The notions concerned can be found in \cite{Loh}.
\begin{definition}[Quasi-equivalence of (generalized) growth functions, \cite{Loh} pp. 171] \
\begin{enumerate}
\item  A generalized growth function is a function of type $\mathbb{R}^+\rightarrow\mathbb{R}^+$ that is nondecreasing.
\item Let $f, g: \mathbb{R}^+\rightarrow\mathbb{R}^+$ be generalized growth functions. We say that $g$ quasi-dominates $f$ if there exist $c, b\in\mathbb{R}^+$ such that
\[
f(r)\leq c\cdot g(c\cdot r+b)+b, \ \ \ \text{for any} \ r\in\mathbb{R}^+.
\]
If $g$ quasi-dominates $f$, then we write $f \prec g$.
\item Two generalized growth functions $f, g: \mathbb{R}^+\rightarrow\mathbb{R}^+$ are quasi-equivalent if both $f \prec g$ and $g \prec f$; if $f$ and $g$ are quasi-equivalent, then we write $f \sim_{QE} g$.
\end{enumerate}
\end{definition}

\begin{definition}[Growth types of finitely generated groups, \cite{Loh} pp. 174] \

Let $G$ be a finitely generated group.
\begin{enumerate}
\item The growth type of $G$ is the (common) quasi-equivalence class of all
growth functions of $G$ with respect to finite generating sets of $G$.
\item The group $G$ is of exponential growth if it has the growth type of the
exponential map $(x\mapsto {\textrm e}^x)$.
\item The group $G$ has polynomial growth if for one (and hence every) finite
generating set $S$ of $G$ there is an $a\in\mathbb{R}^+$ such that ($x\mapsto x^a$) quasi-dominates the growth function of $G$ with respect to $S$.
\item The group $G$ is of intermediate growth if it is neither of exponential nor of polynomial growth.
\end{enumerate}
\end{definition}

If $w_k$ non-increasingly converges to $1$, the weighted sequence $\beta$ induces a generalized growth function $x\mapsto\beta_{[x]}$, where $[x]$ means the integer part of $x$.  Furthermore, if
\[
\sup\limits_{k}(k+1)(w_k-1)=M<\infty,
\]
then
\[
(x\mapsto\beta_{[x]})\prec (x\mapsto x^M),
\]
which implies $\beta$ is of polynomial growth.
Notice that $\beta$ can not be of exponential growth because of $w_k\rightarrow 1$. We could use $w$ to introduce the growth type of the weighted Hardy space $H^2_{\beta}$ with $w_k\rightarrow 1$ as follows.

\begin{definition}
Let $w=\{w_k\}_{k=1}^{\infty}$  be a sequence of positive numbers with $w_k\rightarrow 1$.
\begin{enumerate}
 \item \ If $\sup_{k}(k+1)|w_k-1|<\infty$, we say that the weighted Hardy space $H^2_{\beta}$ is of polynomial growth.
 \item \ If $\sup_{k}(k+1)|w_k-1|=\infty$, we say that the weighted Hardy space $H^2_{\beta}$ is of intermediate growth.
\end{enumerate}
\end{definition}

\begin{remark}
It is easy to see that the condition $\sup_{k}(k+1)|w_k-1|<\infty$ holds if and only if there exists a positive number $M$ such that for each $k\in\mathbb{N}$,
\[
\frac{k+1}{k+M+1}\leq w_k \leq \frac{k+M+1}{k+1}.
\]
Moreover, such weighted Hardy space $H^2_{\beta}$ is said to be of $M$-polynomial growth. Roughly speaking, the polynomial growth condition for a weighted Hardy space $H^2_{\beta}$ implies that $\beta_n$ is controlled by a monomial of $(n+1)$ as an upper bound and a monomial of $\frac{1}{n+1}$ as a lower bound.
\end{remark}

In the present paper, we aim to study on the following three aspects.

($\mathbf{I}$).  An analytic self-map $g: \mathbb{D}\rightarrow \mathbb{D}$ induces a linear operator  $C_g: H^2_{\beta}\rightarrow H^2_{\beta}$, defined by
\[
C_g(f)=f(g),  \ \ \ \text{for any} \ f\in H^2_{\beta}.
\]
Then the operator $C_g$ is said to be a composition operator. The boundedness of the composition operators with symbols of analytic automorphisms is an important problem in analytic function theory and operator theory. There have been obtained many results related to this topic, see \cite{C95} for instance. However, there also have been some unresolved questions. For example, C. Cowen and B. MacCluer proposed a conjecture  \cite{C98} as follows.
\begin{conjecture}[\cite{C98}]\label{Cconj}
	If $\beta_{n}$ is monotone decreasing (or satisfies some other reasonable regularity requirement), $H^2_{\beta}$ is automorphism invariant if and only if there exists a positive integer $n$ so that $(1-z)^{-n}$ is not in $H^2_{\beta}$.
\end{conjecture}

The boundedness of the composition operators with symbols of analytic automorphisms plays an important role in the similarity classification of multiplication operators. It is not difficult to see that, for any $\varphi\in {\textrm Aut(\mathbb{D})}$,  $M_z$ is similar to $M_{\varphi}$, denoted by $M_z\sim M_{\varphi}$, if and only if
the composition operator $C_{\varphi}:H^2_{\beta}\rightarrow H^2_{\beta}$ is an isomorphism. 

On the other hand, weakly homogeneous operator was introduced by Clark and Misra \cite{CM}, which is a generalization of homogeneous operator. A bounded linear operator $T$ on a complex separable Hilbert space is said to be a weakly homogeneous operator if $\sigma(T)\subseteq\overline{\mathbb{D}}$ and $T$ is similar to $\varphi(T)$ for any $\varphi\in {\textrm Aut(\mathbb{D})}$. The boundedness of the composition operators with symbols of analytic automorphisms implies that the multiplication operator $M_z$ is weakly homogeneous. 

In this paper, we aim to study  the boundedness of composition operators with symbols of analytic automorphisms acting on weighted Hardy spaces of polynomial growth. Moreover, we also study the boundedness of composition operators induced by analytic functions in $\textrm{Hol}(\overline{\mathbb{D}})$ on weighted Hardy spaces of polynomial growth, and give counterexamples on a weighted Hardy space of intermediate growth to show the necessity of the setting of polynomial growth condition.

($\mathbf{II}$).  R. Douglas proposed a question in \cite{D07} to ask the similarity of $M_B$ and $\bigoplus_{1}^{m}M_z$ on the classical Bergman space as follows.
\begin{question}[\cite{{D07}}]\label{Dconj}
Is $M_B$ on $L^2_a(D)$ similar to $M_z\otimes \mathbf{I}_m$ on $L^2_a(D)\otimes\mathbb{C}^m$, where $L^2_a(D)$ is the Bergman space and $m$ is the multiplicity of the finite Blaschke product $B(z)$?
\end{question}
Jiang and Li \cite{JL} answered the problem in the affirmative. Furthermore, Jiang and Zheng \cite{JZ} generalized this conclusion on the weighted Bergman spaces. In this paper, we aim to answer the generalized Douglas's question on the weighted Hardy spaces of polynomial growth.

($\mathbf{III}$). We aim to study the Jordan representation theorem of $\textrm{Hol}(\overline{\mathbb{D}})$ on the weighted Hardy spaces of polynomial growth. Furthermore,
we will give the similarity classification of the representation of $\textrm{Hol}(\overline{\mathbb{D}})$, which generalize a result of Jiang and Zheng (Theorem 1.1 in \cite{JZ}) from the weighted Bergman spaces to the weighted Hardy spaces of polynomial growth.

We list our main results in the next section.

\section{Summary of main results}

The first main result in the paper is to obtain the boundedness of the composition operators with symbols of  M\"{o}bius transformations acting on a weighted Hardy space of polynomial growth, which also plays an important role in the whole paper.
\begin{theorem}\label{Mz}
	Let $H^2_{\beta}$ be a weighted Hardy space of polynomial growth induced by the weight sequence $w=\{w_k\}_{k=1}^{\infty}$.  Then for any $\varphi\in {\textrm Aut(\mathbb{D})}$, $M_z\sim M_{\varphi}$, i.e., $M_z$ is weakly homogeneous on $H^2_{\beta}$. In fact, the composition operator $C_{\varphi}:H^2_{\beta}\rightarrow H^2_{\beta}$ is an isomorphism.
\end{theorem}
Correspongding to the conjecture of Cowen and MacCluer, this theorem provides a sufficent condition to $H^2_{\beta}$ being automorphism invariant, in which we 
give a requirement of $w_n$ instead of $\beta_{n}$ (no hypothesis of the monotonicity of $\beta_{n}$ or $w_{n}$). In particular, one could see that if $\beta_{n}$ is monotone decreasing and 
$H^2_{\beta}$ is of $M$-polynomial growth for some positive integer $M$, then $(1-z)^{-(M+1)}$ is not in $H^2_{\beta}$.

Since lots of weighted Hardy spaces of polynomial growth are defined without measures, the measure method is invalid. We have to develop some techniques different from the case of the weighted Bergman spaces or the weighted Dirichlet spaces. Roughly speaking, we will show that $C_{\varphi}$ is a base transformation between two  Riesz bases to prove the above theorem. In section 3,  we study some base properties of the sequence induced by a M\"{o}bius transformation or a finite Blaschke product in base theory, and obtain a series of fundamental results as preliminaries to prove our main results.

We also obtain the boundedness of the composition operators with symbols of analytic functions on the unit closed disk acting on a weighted Hardy space of polynomial growth.
\begin{theorem}\label{CompBHol}
	Let $H^2_{\beta}$ be the weighted Hardy space of polynomial growth induced by a weight sequence $w=\{w_k\}_{k=1}^{\infty}$.
	If $\psi(z)$ is an analytic function on $\overline{\mathbb{D}}$ with $\psi(\mathbb{D})\subseteq \mathbb{D}$, then $C_{\psi}$ is bounded on $H^2_{\beta}$.
\end{theorem}

Furthermore, we give an affirmative answer to the generalized Douglas's question in the setting of weighted Hardy spaces of polynomial growth.

\begin{theorem}\label{MB}
	Let $H^2_{\beta}$ be a weighted Hardy space of polynomial growth, and let $B(z)$ be a finite Blaschke product with order $m$ on $\mathbb{D}$. Then $M_B\sim \bigoplus_{1}^{m}M_z$.
\end{theorem}

In addition, we could give counterexamples on a weighted Hardy space of intermediate growth to show the necessity of the setting of polynomial growth condition.

\begin{theorem}\label{img}
	Let $H^2_{\beta}$ be the weighted Hardy space induced by a weight sequence $w=\{w_k\}_{k=1}^{\infty}$ with $w_k\rightarrow 1$.
	Suppose that there exists a sequence of positive numbers $\{\alpha_j\}_{j=1}^{\infty}$ tending to infinity and a sequence of positive integers $\{n_j\}_{j=1}^{\infty}$ tending to infinity such that, for every $0\leq k\leq n_j$,
	\[
	\frac{\beta_{n_j}}{\beta_k}\geq \frac{\beta_k^{(\alpha_j)}}{\beta_{n_j}^{(\alpha_j)}}.
	\]
	Then, for any $t\in(0,1)$,
	\[
	\frac{\|\varphi^{n_j}_t(z)\|^2_{\beta^{-1}}}{\|z^{n_j}\|^2_{\beta^{-1}}} \rightarrow\infty, \ \ \ \text{as} \ j\rightarrow\infty.
	\]
	Consequently, the composition operators $C_{\varphi_t}: H^2_{\beta^{-1}}\rightarrow H^2_{\beta^{-1}}$ and $C_{\varphi_t}: H^2_{\beta}\rightarrow H^2_{\beta}$ are unbounded. In particular, if
	\[
	\lim\limits_{k\rightarrow\infty}(k+1)(w_k-1)=+\infty \ \ \ \ \text{or} \ \ \ \
	\lim\limits_{k\rightarrow\infty}(k+1)(w_k-1)=-\infty,
	\]
	then the conclusion holds.
\end{theorem}

 Together with Theorem \ref{MB} and the technique of K-theory of Banach algebras, we could obtain the Jordan representation theorem and similarity classification of the representation of $\textrm{Hol}(\overline{\mathbb{D}})$ on the weighted Hardy spaces of polynomial growth.

\begin{theorem}[Jordan representation theorem of $\textrm{Hol}(\overline{\mathbb{D}})$]\label{JPT}
	Given any $f\in \textrm{Hol}(\overline{\mathbb{D}})$. There exist a unique positive integer $m$ and an analytic function $h\in \textrm{Hol}(\overline{\mathbb{D}})$ with $\{M_h \}_{H^2_{\beta}}' =H^{\infty}_{\beta}$, such that
	\[
	M_f\sim\bigoplus\limits_1^m M_h,
	\]
	where $h$ is unique in the sense of analytic automorphism group action, i.e., if there exists another $g\in \textrm{Hol}(\overline{\mathbb{D}})$ with $\{M_g \}_{H^2_{\beta}}' =H^{\infty}_{\beta}$ such that
	\[
	M_f\sim\bigoplus\limits_1^m M_g,
	\]
	then there is a M\"{o}bius transformation $\varphi\in \textrm{Aut}(\mathbb{D})$ such that $g=h\circ \varphi$.
\end{theorem}

\begin{theorem}\label{Simlar}
	Let $f_1, f_2\in \textrm{Hol}(\overline{\mathbb{D}})$. Then, $M_{f_1}$ is similar to $M_{f_2}$ on $H^2_{\beta}$ if and only if there are two finite Blaschke products $B_1$ and
	$B_2$ with the same order and a function $h \in \textrm{Hol}(\overline{\mathbb{D}})$ such that
	\[
	f_1=h\circ B_1 \ \ \ \text{and} \ \ \ f_2=h\circ B_2.
	\]
\end{theorem}

We will give a series of fundamental results as preliminaries in section 3, and then give the proofs of Theorem \ref{Mz}, Theorem \ref{CompBHol}, Theorem \ref{MB} and Theorem \ref{img} in section 4, and  give the proofs of Theorem \ref{JPT} and Theorem \ref{Simlar} in section 5.

\section{Bases theory and transformation operators}

In this section, let us review some basic concepts in bases theory on Hilbert space firstly (we refer to \cite{Nik}, more generally, one can see these concepts on Banach space in \cite{Sing}). Let $\mathcal{H}$ be a complex separable Hilbert space and $\mathfrak{X}=\{x_n\}_{n=0}^{\infty}$ be a sequence of vectors in $\mathcal{H}$. The sequence $\mathfrak{X}$ is said to be total, if $\mathfrak{X}$ spans the whole space $\mathcal{H}$, i.e.,
\[
\overline{{\textrm span}\mathfrak{X}}=\overline{\left\{\sum\limits_{n=0}^{\infty}c_nx_n; \ x_n\in \mathfrak{X}, \ c_n\in \mathbb{C} \ \text{is zero except finite items}\right\}}=\mathcal{H}.
\]
$\mathfrak{X}$ is said to be finitely linear independent, if its arbitrary finite subsequence is linear independent. $\mathfrak{X}$ is said to be a Schauder base, if for every $x\in \mathcal{H}$, there exists a unique sequence of complex numbers $\{c_n\}_{n=0}^{\infty}$ such that
\[
x=\sum\limits_{n=0}^{\infty}c_nx_n,\ \ \text{i.e.}, \ \ \ \lim\limits_{n\rightarrow\infty}\|x-\sum\limits_{k=0}^{n}c_kx_k\|=0.
\]
A Schauder base $\mathfrak{X}=\{x_n\}_{n=0}^{\infty}$ is said to be a normalized base if $\|x_n\|=1$ for all $n=0,1,\ldots$.
Moreover, a Schauder base $\mathfrak{X}=\{x_n\}_{n=0}^{\infty}$ is called a bounded (or quasinormed) base if
\[
0<\inf\limits_{0\leq n <\infty}\|x_n\|\leq \sup\limits_{0\leq n <\infty}\|x_n\|<\infty.
\]
A Schauder base $\mathfrak{X}=\{x_n\}_{n=0}^{\infty}$ is said to be an unconditional base, if every convergent series of the form $\sum_{n=0}^{\infty}c_nx_n$ is unconditional convergent, i.e., independently of order. Furthermore, A Schauder base is called conditional base, if it is not an unconditional base. $\mathfrak{X}=\{x_n\}_{n=0}^{\infty}$ is said to be a Riesz base, if there exists a bounded invertible operator $V$ such that $\{V(x_n)\}_{n=0}^{\infty}$ is an orthonormal base, i.e., $V(x_m)$ is orthogonal to $V(x_n)$ for all $m\neq n$, and $\|V(x_n)\|=1$ for all $n=0,1,\ldots$.

Let $\mathfrak{X}=\{x_n\}_{n=0}^{\infty}$ and $\mathfrak{Y}=\{y_n\}_{n=0}^{\infty}$ be two sequences of vectors in a complex separable Hilbert space $\mathcal{H}$.
Define
\[
\langle  \mathfrak{X}, \mathfrak{Y}\rangle_{\mathcal{H}}=(\langle  x_j, y_i\rangle_{\mathcal{H}})_{i,j}.
\]

In particular, the Gram matrix of $\mathfrak{X}$ on $\mathcal{H}$ is defined by
\[
\Gamma_{\mathcal{H}}(\mathfrak{X})=\langle  \mathfrak{X}, \mathfrak{X}\rangle_{\mathcal{H}}=(\langle  x_i, x_j\rangle_{\mathcal{H}})_{i,j}.
\]

Next, we will study some base properties of the sequence induced by a M\"{o}bius transformation or a finite Blaschke product in bases theory. It is useful of a class of geometric operators introduced by M. Cowen and R. Douglas \cite{CD} as follows.
\begin{definition}[\cite{CD}]
For $\Omega$ a connected open subset of $\mathbb{C}$ and $n$ a
positive integer, let $\mathbf {B}_{n}(\Omega)$ denote the
operators $T$ in $\mathcal{L}(\mathcal{H})$ which satisfy:

(a) $\Omega \subseteq \sigma(T)=\{\omega \in \mathbb{C}: \ T-\omega
 \ \text{not \ invertible}\}$;

(b) ${\textrm Ran}(T-\omega)=\mathcal {H}$ for every $\omega$ in $\Omega$;

(c) $\overline{{\textrm span}\{{\textrm Ker}(T-\omega); {\omega\in \Omega}\}}=\mathcal {H}$;

(d) $\dim {\textrm Ker}(T-\omega)=n$ for every $\omega$ in $\Omega$.
\end{definition}
\begin{lemma}\label{leftin}
Suppose that $H^2_{\beta}$ is the weighted Hardy space induced by a weight sequence $w=\{w_k\}_{k=1}^{\infty}$ with $w_k\rightarrow 1$. Let $B(z)=\prod^m_{j=1}\frac{z_{j}-z}{1-\overline{z_{j}}z}$ be a Blaschke product on $\mathbb{D}$ with order $m$. Denote
\[
\overline{B}(z)=\overline{B(\overline{z})}=\sum\limits^{\infty}_{k=0}\overline{\widehat{B}(k)}z^k.
\]
Let $S$ be the standard left inverse of the multiplication operator $M_z$ on $H^2_{\beta}$, defined by
\[
S(f(z))=\frac{f(z)-f(0)}{z}, \ \ \ \text{for every} \ f(z)\in H^2_{\beta}.
\]
Then $\overline{B}(S)\in \mathbf {B}_m(\mathbb{D})$ and $\overline{B}(S)B(M_z)={\textbf I}$, where ${\textbf I}$ means the identity operator.
\end{lemma}

\begin{proof}
Notice that $S\in\mathbf {B}_1(\mathbb{D})$ and $\overline{B}(z)$ is also a Blaschke product on $\mathbb{D}$ with order $m$. Then it is easy to see $\overline{B}(S)\in \mathbf {B}_m(\mathbb{D})$.

For any $\alpha\in\mathbb{D}$, let $\varphi_{\alpha}=\frac{\alpha-z}{1-\overline{\alpha}z}$. Then, it follows from $SM_z={\textbf I}$ that
\begin{align*}
\overline{\varphi_{\alpha}}(S)\varphi_{\alpha}(M_z)
&=(\overline{\alpha}-S)\sum\limits_{k=0}^{\infty}(\alpha S)^k(\alpha-M_z)({\textbf I}-\overline{\alpha}M_z)^{-1} \\
&=-(\overline{\alpha}-S)M_z({\textbf I}-\overline{\alpha}M_z)^{-1} \\
&={\textbf I}.
\end{align*}
Since
$B(z)=\prod^m_{j=1}\varphi_{z_{j}}(z)$,
we have
\[
\overline{B}(S)B(M_z)=\prod\limits_{j=1}^{m}\overline{\varphi_{z_{j}}}(S)\prod\limits_{j=1}^{m}\varphi_{z_{j}}(M_z)={\textbf I}.
\]
\end{proof}

\begin{proposition}\label{Btotal}
Let $H^2_{\beta}$ be the weighted Hardy space induced by a weight sequence $w=\{w_k\}_{k=1}^{\infty}$ with $w_k\rightarrow 1$, and let $B(z)=\prod^m_{j=1}\frac{z_{j}-z}{1-\overline{z_{j}}z}$ be a Blaschke product on $\mathbb{D}$ with order $m$.
Suppose that $\{f_1, \cdots, f_m\}$ is a base of ${\textrm Ker}(\overline{B}(S))$.  Then
\[
\{f_1B^n, \cdots, f_mB^n; n=0,1,2,\ldots\}
\]
is a total and finitely linear independent sequence in $H^2_{\beta}$.
\end{proposition}

\begin{proof}
First, we will show that, for any $N=0,1,2,\ldots$,
\[
\{f_1B^n, \cdots, f_mB^n; n=0,1,\ldots, N\}
\]
is linear independent. Assume
\[
\sum\limits_{n=0}^N\sum\limits_{j=1}^m\lambda_{nj}f_j(z)B^n(z)=0.
\]
Then, by Lemma \ref{leftin},
\[
0=\overline{B}^N(S)(\sum\limits_{n=0}^N\sum\limits_{j=1}^m\lambda_{nj}f_j(z)B^n(z))=\sum\limits_{j=1}^m\lambda_{Nj}f_j(z).
\]
Since $\{f_1, \cdots, f_m\}$ is a base of ${\textrm Ker}(\overline{B}(S))$, we get
\[
\lambda_{N1}=\lambda_{N2}=\cdots=\lambda_{Nm}=0.
\]
In this manner, one can see
\[
\lambda_{nj}=0, \ \ \text{for all} \ n=0,1,\ldots, N \ \text{and} \ j=1,2,\ldots, m.
\]

In addition, for any $n=0,1,\ldots,N$ and any $j=1,2,\ldots, m$,
\[
\overline{B}^{N+1}(S)(f_j(z)B^n(z))=0.
\]
Then, by ${\dim}{\textrm Ker}(\overline{B}^{N+1}(S))=m(N+1)$, the sequence
\[
\{f_1B^n, \cdots, f_mB^n; n=0,1,\ldots, N\}
\]
is a base of ${\textrm Ker}(\overline{B}^{N+1}(S))$.

Therefore, by
\[
\overline{{\textrm span}\{{\textrm Ker}(\overline{B}^N(S)); N=1,2,\ldots \}}=H^2_{\beta},
\]
the sequence
\[
\{f_1B^n, \cdots, f_mB^n; n=0,1,2,\ldots\}
\]
is total and finitely linear independent in $H^2_{\beta}$.
\end{proof}

Furthermore, we could give a concrete base of ${\textrm Ker}(\overline{B}(S))$ if the Blaschke product $B$ has distinct zero points.

\begin{lemma}[\cite{JZ}]\label{mmatrix}
Let $z_1, \cdots, z_{m}$ be $m$ distinct points in $\mathbb{D}$.
Then the matrix
\[
\begin{bmatrix}
\frac{1}{1-|z_1|^2} & \frac{1}{1-\overline{z_1}z_2} & \frac{1}{1-\overline{z_1}z_3}  & \cdots & \frac{1}{1-\overline{z_1}z_m} \\
\frac{1}{1-\overline{z_2}z_1}   & \frac{1}{1-|z_2|^2} & \frac{1}{1-\overline{z_2}z_3}  & \cdots & \frac{1}{1-\overline{z_2}z_m}  \\
\frac{1}{1-\overline{z_3}z_1}  & \frac{1}{1-\overline{z_3}z_2}   & \frac{1}{1-|z_3|^2}  & \cdots  & \frac{1}{1-\overline{z_3}z_m} \\
\vdots   & \vdots & \vdots & \ddots &\vdots \\
\frac{1}{1-\overline{z_m}z_1}   & \frac{1}{1-\overline{z_m}z_2} & \frac{1}{1-\overline{z_m}z_3} & \cdots &\frac{1}{1-|z_m|^2} \\
\end{bmatrix}
\]
is invertible.
\end{lemma}

\begin{theorem}\label{disBtotal}
Let $H^2_{\beta}$ be the weighted Hardy space induced by a weight sequence $w=\{w_k\}_{k=1}^{\infty}$ with $w_k\rightarrow 1$, and let $B(z)=\prod^m_{j=1}\frac{z_{j}-z}{1-\overline{z_{j}}z}$ be a Blaschke product on $\mathbb{D}$ with $m$ distinct zero points.
Then
\[
\{\frac{B^n(z)}{1-\overline{z_{1}}z}, \cdots, \frac{B^n(z)}{1-\overline{z_{m}}z}; n=0,1,2,\ldots\}
\]
is a total and finitely linear independent sequence in $H^2_{\beta}$.
\end{theorem}

\begin{proof}
Following from Proposition \ref{Btotal}, it suffices to prove the finite sequence $\{\frac{1}{1-\overline{z_{j}}z}\}_{j=1}^m$ is a base of ${\textrm Ker}(\overline{B}(S))$.
Notice that $\{\frac{1}{1-\overline{z_{j}}z}\}_{j=1}^m$ is  linear independent in $H^2_{\beta}$ if and only if so is in the classical Hardy space $H^2$. By Lemma \ref{mmatrix}, the sequence $\{\frac{1}{1-\overline{z_{j}}z}\}_{j=1}^m$ is  linear independent in $H^2_{\beta}$. In addition, for $j=1,2,\ldots, m$,
\begin{align*}
\overline{B}(S)(\frac{1}{1-\overline{z_{j}}z})&=(\prod\limits_{i\neq j}\overline{\varphi_{z_{i}}}(S)) \overline{\varphi_{z_{j}}}(S)(\frac{1}{1-\overline{z_{j}}z}) \\
&=(\prod\limits_{i\neq j}\overline{\varphi_{z_{i}}}(S)) ({\textbf I}-z_jS)^{-1}(\overline{z_{j}}-S)(\sum\limits_{k=0}^{\infty} (\overline{z_{j}}z)^k) \\
&=0.
\end{align*}
Then, by ${\dim}{\textrm Ker}(\overline{B}(S))=m$, the sequence $\{\frac{1}{1-\overline{z_{j}}z}\}_{j=1}^m$ is a base of ${\textrm Ker}(\overline{B}(S))$.
\end{proof}

\begin{corollary}\label{twoBtotal}
Let $H^2_{\beta}$ be the weighted Hardy space induced by a weight sequence $w=\{w_k\}_{k=1}^{\infty}$ with $w_k\rightarrow 1$. Let $\varphi_{z_0}(z)=\frac{z_0-z}{1-\overline{z_0}z}$, $z_0\in\mathbb{D}\setminus\{0\}$, and $B(z)=z\varphi_{z_0}(z)$.
Then, the sequences
\[
\{B^n(z), zB^n(z); n=0,1,2,\ldots\} \ \ \ \text{and} \ \ \ \{B^n(z), \varphi_{z_0}(z)B^n(z); n=0,1,2,\ldots\}
\]
are total and finitely linear independent in $H^2_{\beta}$, respectively.
\end{corollary}

\begin{proof}
According to Theorem \ref{disBtotal}, the sequence
\[
\{B^n(z), \frac{B^n(z)}{1-\overline{z_{0}}z}; n=0,1,2,\ldots\}
\]
is total and finitely linear independent in $H^2_{\beta}$. Since
{\small
\[
{\textrm span}\{1, \frac{1}{1-\overline{z_{0}}z}\}=\{\frac{az+b}{1-\overline{z_{0}}z}; a,b\in\mathbb{C}\}={\textrm span}\{\frac{1}{1-\overline{z_{0}}z}, \frac{z}{1-\overline{z_{0}}z}\}={\textrm span}\{1, \frac{z_0-z}{1-\overline{z_{0}}z}\},
\]
}
it follows from Proposition \ref{Btotal} that the sequences
\[
\{\frac{B^n(z)}{1-\overline{z_{0}}z}, \frac{zB^n(z)}{1-\overline{z_{0}}z}; n=0,1,2,\ldots\} \ \ \ \text{and} \ \ \ \{B^n(z), \varphi_{z_0}B^n(z); n=0,1,2,\ldots\}
\]
are total and finitely linear independent in $H^2_{\beta}$, respectively. Notice that
\[
M_{\frac{1}{1-\overline{z_{0}}z}}:H^2_{\beta}\rightarrow H^2_{\beta}
\]
is a bounded invertible operator. Thus, the sequence
\[
\{B^n(z), zB^n(z); n=0,1,2,\ldots\}
\]
is also total and finitely linear independent in $H^2_{\beta}$.
\end{proof}

In particular, we could obtain that the powers of a M\"{o}bius transformation constitute a Schauder base by the uniqueness of Taylor expansion.
\begin{proposition}\label{Schauder}
Let $H^2_{\beta}$ be the weighted Hardy space induced by a weight sequence $w=\{w_k\}_{k=1}^{\infty}$ with $w_k\rightarrow 1$. Let $\varphi_{z_0}(z)=\frac{z_0-z}{1-\overline{z_0}z}$, $z_0\in\mathbb{D}$, be a M\"{o}bius transformation on $\mathbb{D}$.
Then, the sequence $\{\varphi_{z_0}^n\}_{n=0}^{\infty}$ is a Schauder base of $H^2_{\beta}$.
\end{proposition}

\begin{proof}
It follows from Theorem \ref{disBtotal} that the sequence $\{\varphi_{z_0}^n\}_{n=0}^{\infty}$ is total and finitely linear independent in $H^2_{\beta}$. For any $f\in H^2_{\beta}\subseteq \textrm{Hol}(\mathbb{D})$, we could write
\[
f(z)=\sum\limits_{k=0}^{\infty}\lambda_k\varphi_{z_0}^k(z).
\]
Let $w=\varphi_{z_0}(z)$. Then
\[
f(\varphi_{z_0}(w))=\sum\limits_{k=0}^{\infty}\lambda_kw^k, \ \ \  \text{for all} \ w\in\mathbb{D}.
\]
According to the uniqueness of the Taylor expansion of the function $f\in \textrm{Hol}(\mathbb{D})$, the sequence $\{\lambda_k\}_{k=0}^{\infty}$ is unique.
Hence, the function $f(z)$ has unique coefficients $\lambda_k$. This completes the proof.
\end{proof}

Now we focus on the characterization of Riesz base in a complex separable Hilbert space. It is well-known that Riesz base and bounded unconditional base are equivalent in a complex separable Hilbert space. However, there exist conditional bases in Hilbert space \cite{B48}, and furthermore, Olevskii gave a spectral characterization the transformation operator from an orthonormal base to a conditional base in \cite{Ole}. So it is worthy of describing Riesz bases in a computable manner. Bari had used Gram matrix to provide a sufficient and necessary condition to Riesz base \cite{B51} (see also in \cite{Nik}).

\begin{theorem}[Bari's Theorem]\label{Bari}
Let $\mathfrak{X}$  be a total sequence in a complex separable Hilbert space $\mathcal{H}$. Then $\mathfrak{X}$ is a Riesz base if and only if the Gram matrix $\Gamma_{\mathcal{H}}(\mathfrak{X})$ is a bounded invertible linear operator on $l^2$.
\end{theorem}

Inspired by Olevskii's idea in \cite{Ole}, we will introduce the notion of transformation operator of a sequence in $H^2_{\beta}$ to characterize Riesz bases.

Let $\mathfrak{F}=\{f_n\}_{n=0}^{\infty}$ be a sequence of vectors in a weighted Hardy space $H^2_{\beta}$. Then $\mathfrak{F}$ induces a linear operator $X_{\mathfrak{F}}$ on $H^2_{\beta}$, defined by
\[
X_{\mathfrak{F}}(z^n)=f_n(z), \ \ \ \text{for} \ n=0,1,\ldots.
\]
Furthermore, the linear operator $X_{\mathfrak{F}}$ has a matrix representation under the orthogonal base $\{z^n\}_{n=0}^{\infty}$ as follows
\[
X_{\mathfrak{F}}=\begin{bmatrix}
\widehat{f_0}(0) & \widehat{f_1}(0) & \widehat{f_2}(0)  & \cdots & \widehat{f_k}(0) & \cdots \\
\widehat{f_0}(1)   & \widehat{f_1}(1) & \widehat{f_2}(1)  & \cdots & \widehat{f_k}(1) & \cdots \\
\widehat{f_0}(2)  & \widehat{f_1}(2)   & \widehat{f_2}(2)  & \cdots  & \widehat{f_k}(2) & \cdots \\
\vdots   & \vdots & \vdots & \ddots &\vdots  &\vdots \\
\widehat{f_0}(k)   & \widehat{f_1}(k) & \widehat{f_2}(k) & \cdots & \widehat{f_k}(k) & \cdots \\
\vdots   & \vdots & \vdots&\vdots  &\vdots & \ddots
\end{bmatrix}
\]
We say that the operator $X_{\mathfrak{F}}$ is the transformation operator of the sequence $\mathfrak{F}$. In particular, if $\mathfrak{F}=\{g^n\}_{n=0}^{\infty}$ for some $g\in H^{\infty}_{\beta}$ with $g(\mathbb{D})\subseteq\mathbb{D}$, then the transformation operator $X_{\mathfrak{F}}$ is just the composition operator $C_g$.
Notice that the transformation operator $X_{\mathfrak{F}}$ may be not bounded on $H^2_{\beta}$ in general.

Define the operator $D_{\beta}: H^2_{\beta}\rightarrow H^2$ by
\[
D_{\beta}(z^k)=\beta_kz^k  \ \ \ \text{for} \ n=0,1,\ldots.
\]
As well known, $D_{\beta}$ is an isometry isomorphism. Moreover, $D_{\beta}$ and $D^{-1}_{\beta}$ has matrix representations under the orthogonal base $\{z^n\}_{n=0}^{\infty}$ as follows
\[
D_{\beta}=\begin{bmatrix}
\beta_0 & 0  & \cdots & 0 & \cdots \\
0   & \beta_1   & \cdots & 0 & \cdots \\
\vdots   & \vdots  & \ddots &\vdots  &\vdots \\
0   & 0  & \cdots &\beta_k & \cdots \\
\vdots   & \vdots &\vdots  &\vdots & \ddots
\end{bmatrix}, \ \
D^{-1}_{\beta}=D_{\beta^{-1}}=\begin{bmatrix}
\frac{1}{\beta_0} & 0   & \cdots & 0 & \cdots \\
0   & \frac{1}{\beta_1}  & \cdots & 0 & \cdots \\
\vdots   & \vdots  & \ddots &\vdots  &\vdots \\
0   & 0  & \cdots &\frac{1}{\beta_k} & \cdots \\
\vdots   & \vdots &\vdots  &\vdots & \ddots
\end{bmatrix}.
\]
Furthermore, $D_{\beta}X_{\mathfrak{F}}D^{-1}_{\beta}$ is an operator on the classical Hardy space $H^2$.

For a sequence $\mathfrak{F}=\{f_n\}_{n=0}^{\infty}$ in $H^2_{\beta}$, we always denote $\mathfrak{F}_{\beta}=\{\frac{f_n}{\beta_n}\}_{n=0}^{\infty}$. Now, we could describe the boundedness of the Gram matrix of $\mathfrak{F}_{\beta}$ on $H^2_{\beta}$ by the transformation operator $X_{\mathfrak{F}}$. Since
\begin{align*}
\Gamma_{H^2_{\beta}}(\mathfrak{F}_{\beta})&=\langle  \mathfrak{F}_{\beta}, \mathfrak{F}_{\beta}\rangle_{H^2_{\beta}}  \\
&=D^{-1}_{\beta}X^*_{\mathfrak{F}} D_{\beta}^2 X_{\mathfrak{F}}D^{-1}_{\beta} \\
&=(D^{-1}_{\beta}X^*_{\mathfrak{F}} D_{\beta})(D_{\beta} X_{\mathfrak{F}}D^{-1}_{\beta}).
\end{align*}
It is easy to see the following statements are equivalent.
\begin{enumerate}
 \item \ $\Gamma_{H^2_{\beta}}(\mathfrak{F}_{\beta})$ is a bounded operator from $H^2$ to $H^2$.
 \item \ $D_{\beta}X_{\mathfrak{F}}D^{-1}_{\beta}$ is a bounded operator from $H^2$ to $H^2$.
 \item \ $X_{\mathfrak{F}}$ is a bounded operator from $H^2_{\beta}$ to $H^2_{\beta}$.
\end{enumerate}
Similarly, the following statements are also equivalent.
\begin{enumerate}
 \item \ $\Gamma_{H^2_{\beta}}(\mathfrak{F}_{\beta})$ is an invertible operator on $H^2$.
 \item \ $D_{\beta}X_{\mathfrak{F}}D^{-1}_{\beta}$ is a lower bounded operator on $H^2$, i.e., there exists a constant $K>0$ such that for any $f\in H^2$,
     \[
     \|D_{\beta}X_{\mathfrak{F}}D^{-1}_{\beta}(f)\|_{H^2}\geq K\|f\|_{H^2}.
     \]
 \item \ $X_{\mathfrak{F}}$ is a lower bounded operator on $H^2_{\beta}$,  i.e., there exists a constant $K'>0$ such that for any $f\in H^2_{\beta}$,
     \[
     \|X_{\mathfrak{F}}(f)\|_{H^2_{\beta}}\geq K'\|f\|_{H^2_{\beta}}.
     \]
\end{enumerate}
Then, we could rewrite the Bari's Theorem (Theorem \ref{Bari}) as follows.
\begin{lemma}\label{RBari}
Let $\mathfrak{F}=\{f_n\}_{n=0}^{\infty}$  be a total sequence in $H^2_{\beta}$. Then $\mathfrak{F}_{\beta}$ is a Riesz base of $H^2_{\beta}$ if and only if $D_{\beta}X_{\mathfrak{F}}D^{-1}_{\beta}$ is a bounded and lower bounded operator on $H^2$.
\end{lemma}

Then, we could match some discussion in $H^2_{\beta}$ with the ones in $H^2_{\beta^{-1}}$ as follows.

\begin{proposition}\label{CorA}
Let $H^2_{\beta}$ be the weighted Hardy space induced by a weight sequence $w=\{w_k\}_{k=1}^{\infty}$ with $w_k\rightarrow 1$. Let $\varphi_{z_0}(z)=\frac{z_0-z}{1-\overline{z_0}z}$, $z_0\in\mathbb{D}$, be a M\"{o}bius transformation on $\mathbb{D}$ and denote $\mathfrak{F}=\{\varphi_{z_0}^n\}_{n=0}^{\infty}$.
Then, the following are equivalent.
\begin{enumerate}
 \item \ $C_{\varphi_{z_0}}$ is a bounded operator on $H^2_{\beta}$.
 \item \ $C_{\varphi_{z_0}}$ is a bounded invertible operator on $H^2_{\beta}$.
 \item \ $\mathfrak{F}_{\beta}$ is a Riesz base of $H^2_{\beta}$.
 \item \ $\mathfrak{F}_{\beta^{-1}}$ is a Riesz base of $H^2_{\beta^{-1}}$.
 \item \ $C_{\varphi_{z_0}}$ is a bounded invertible operator on $H^2_{\beta^{-1}}$.
 \item \ $C_{\varphi_{z_0}}$ is a bounded operator on $H^2_{\beta^{-1}}$.
\end{enumerate}
\end{proposition}

\begin{proof}
Firstly, since $C_{\varphi_{z_0}}=X_{\mathfrak{F}}$ and $C_{\varphi_{z_0}}\circ C_{\varphi_{z_0}}={\textbf I}$, we obtain $(1)$, $(2)$ and $(3)$ are equivalent. Similarly, $(4)$, $(5)$ and $(6)$ are equivalent. It suffices to prove $(3)\Leftrightarrow(4)$.

Now assume that $\mathfrak{F}_{\beta}$ is a Riesz base of $H^2_{\beta}$. By Proposition \ref{Btotal} or Proposition \ref{Schauder}, $\mathfrak{F}_{\beta}$ is a total sequence in $H^2_{\beta}$. Then, it follows from Lemma \ref{RBari} that $D_{\beta}X_{\mathfrak{F}}D^{-1}_{\beta}$ is a bounded invertible operator on $H^2$. Notice that $\{\frac{\sqrt{1-|z_0|^2}}{1-\overline{z_0}z}\varphi_{z_0}^n(z)\}_{n=0}^{\infty}$ is an orthonormal base in the classical Hardy space $H^2$. Then,
\begin{align*}
&(D_{\beta}X^*_{\mathfrak{F}}D^{-1}_{\beta})(D_{\beta}M^*_{\frac{\sqrt{1-|z_0|^2}}{1-\overline{z_0}z}}
M_{\frac{\sqrt{1-|z_0|^2}}{1-\overline{z_0}z}}D^{-1}_{\beta})(D_{\beta}X_{\mathfrak{F}}D^{-1}_{\beta}) \\
=& D_{\beta}X^*_{\mathfrak{F}}M^*_{\frac{\sqrt{1-|z_0|^2}}{1-\overline{z_0}z}}
M_{\frac{\sqrt{1-|z_0|^2}}{1-\overline{z_0}z}}X_{\mathfrak{F}}D^{-1}_{\beta} \\
=& D_{\beta}\langle\frac{\sqrt{1-|z_0|^2}}{1-\overline{z_0}z}{\mathfrak{F}}, \frac{\sqrt{1-|z_0|^2}}{1-\overline{z_0}z}\mathfrak{F}\rangle_{H^2}D^{-1}_{\beta} \\
=&{\textbf I}.
\end{align*}
Since
\[
D_{\beta}X_{\mathfrak{F}}D^{-1}_{\beta} \ \ \ \text{and} \ \ \ D_{\beta}M^*_{\frac{\sqrt{1-|z_0|^2}}{1-\overline{z_0}z}}M_{\frac{\sqrt{1-|z_0|^2}}{1-\overline{z_0}z}}D^{-1}_{\beta}
\]
are bounded invertible operators on $H^2$, we obtain that $D_{\beta}X^*_{\mathfrak{F}}D^{-1}_{\beta}$ is a bounded invertible operator on $H^2$ and so is
\[
D^{-1}_{\beta}X_{\mathfrak{F}}D_{\beta}=(D_{\beta}X^*_{\mathfrak{F}}D^{-1}_{\beta})^*.
\]
Applying Lemma \ref{RBari} again, we obtain $(3)\Rightarrow (4)$. In the same way, one can see $(4)\Rightarrow(3)$.
\end{proof}

Notice that $M_z\sim M_{\varphi_{z_0}}$ if and only if $C_{\varphi_{z_0}}$ is a bounded invertible operator. Then we obtain the following corollary immediately.
\begin{corollary}\label{SimCorA}
Let $H^2_{\beta}$ be the weighted Hardy space induced by a weight sequence $w=\{w_k\}_{k=1}^{\infty}$ with $w_k\rightarrow 1$. Let $\varphi_{z_0}(z)=\frac{z_0-z}{1-\overline{z_0}z}$, $z_0\in\mathbb{D}$, be a M\"{o}bius transformation on $\mathbb{D}$.
Then, $M_z\sim M_{\varphi_{z_0}}$ on $H^2_{\beta}$ if and only if $M_z\sim M_{\varphi_{z_0}}$ on $H^2_{\beta^{-1}}$.
\end{corollary}

Now let $B(z)=\prod^m_{j=1}\frac{z_{j}-z}{1-\overline{z_{j}}z}$ be a Blaschke product on $\mathbb{D}$ with $m$ distinct zero points. Denote $\mathfrak{F}=\{B^n\}_{n=0}^{\infty}$ and
\[
\widetilde{\mathfrak{F}_{\beta}}=\bigcup\limits_{j=1}^{m}\frac{\mathfrak{F}_{\beta}}{1-\overline{z_{j}}z}=\{\frac{1}{1-\overline{z_{1}}z}\frac{B^n}{\beta_n}, \cdots, \frac{1}{1-\overline{z_{m}}z}\frac{B^n}{\beta_n}; n=0,1,2,\ldots\}.
\]

\begin{proposition}\label{Btotilde}
Let $B(z)=\prod^m_{j=1}\frac{z_{j}-z}{1-\overline{z_{j}}z}$ be a Blaschke product on $\mathbb{D}$ with $m$ distinct zero points. Then, $\Gamma_{H^2_{\beta}}(\mathfrak{F}_{\beta})$ is a bounded operator on $H^2$ if and only if $\Gamma_{H^2_{\beta}}(\widetilde{\mathfrak{F}_{\beta}})$ is a bounded operator on $H^2$.
\end{proposition}

\begin{proof}
One can see
\begin{align*}
\Gamma_{H^2_{\beta}}(\widetilde{\mathfrak{F}_{\beta}})&=\langle  \widetilde{\mathfrak{F}_{\beta}}, \widetilde{\mathfrak{F}_{\beta}}\rangle_{H^2_{\beta}} \\
&=\begin{bmatrix}
\langle \frac{\mathfrak{F}_{\beta}}{1-\overline{z_{1}}z}, \frac{\mathfrak{F}_{\beta}}{1-\overline{z_{1}}z} \rangle_{H^2_{\beta}} & \langle \frac{\mathfrak{F}_{\beta}}{1-\overline{z_{1}}z}, \frac{\mathfrak{F}_{\beta}}{1-\overline{z_{2}}z} \rangle_{H^2_{\beta}}   & \cdots & \langle \frac{\mathfrak{F}_{\beta}}{1-\overline{z_{1}}z}, \frac{\mathfrak{F}_{\beta}}{1-\overline{z_{m}}z} \rangle_{H^2_{\beta}} \\
\langle \frac{\mathfrak{F}_{\beta}}{1-\overline{z_{2}}z}, \frac{\mathfrak{F}_{\beta}}{1-\overline{z_{1}}z} \rangle_{H^2_{\beta}}   & \langle \frac{\mathfrak{F}_{\beta}}{1-\overline{z_{2}}z}, \frac{\mathfrak{F}_{\beta}}{1-\overline{z_{2}}z} \rangle_{H^2_{\beta}}   & \cdots & \langle \frac{\mathfrak{F}_{\beta}}{1-\overline{z_{2}}z}, \frac{\mathfrak{F}_{\beta}}{1-\overline{z_{m}}z} \rangle_{H^2_{\beta}}  \\
\vdots   & \vdots & \ddots &\vdots \\
\langle \frac{\mathfrak{F}_{\beta}}{1-\overline{z_{m}}z}, \frac{\mathfrak{F}_{\beta}}{1-\overline{z_{1}}z} \rangle_{H^2_{\beta}}   & \langle \frac{\mathfrak{F}_{\beta}}{1-\overline{z_{m}}z}, \frac{\mathfrak{F}_{\beta}}{1-\overline{z_{2}}z} \rangle_{H^2_{\beta}}  & \cdots & \langle \frac{\mathfrak{F}_{\beta}}{1-\overline{z_{m}}z}, \frac{\mathfrak{F}_{\beta}}{1-\overline{z_{m}}z} \rangle_{H^2_{\beta}} \\
\end{bmatrix}.
\end{align*}
Notice that for any $i,j=1,2,\ldots,\ldots,m$
{\[
\langle \frac{\mathfrak{F}_{\beta}}{1-\overline{z_{i}}z}, \frac{\mathfrak{F}_{\beta}}{1-\overline{z_{j}}z} \rangle_{H^2_{\beta}}=
(D^{-1}_{\beta}X^*_{\mathfrak{F}}D_{\beta})(D^{-1}_{\beta}M_{\frac{1}{1-\overline{z_{j}}z}}^*D_{\beta})(D_{\beta}
M_{\frac{1}{1-\overline{z_{i}}z}}D^{-1}_{\beta})(D_{\beta}X_{\mathfrak{F}}D^{-1}_{\beta}).
\]}
By the boundedness of $D^{-1}_{\beta}M_{\frac{1}{1-\overline{z_{j}}z}}^*D_{\beta}$ and $D_{\beta}M_{\frac{1}{1-\overline{z_{i}}z}}D^{-1}_{\beta}$, $\langle \frac{\mathfrak{F}_{\beta}}{1-\overline{z_{i}}z}, \frac{\mathfrak{F}_{\beta}}{1-\overline{z_{j}}z} \rangle_{H^2_{\beta}}$ is bounded on $H^2$ if and only if $D_{\beta}X_{\mathfrak{F}}D^{-1}_{\beta}$ is bounded on $H^2$. Therefore, $\Gamma_{H^2_{\beta}}(\mathfrak{F}_{\beta})$ is a bounded operator on $H^2$ if and only if $\Gamma_{H^2_{\beta}}(\widetilde{\mathfrak{F}_{\beta}})$ is a bounded operator on $H^2$.
\end{proof}

At the end of this section, we extend the boundedness of some composition operators from one weighted Hardy space to a wider class of weighted Hardy spaces.

\begin{theorem}[\cite{C90}]\label{Controll}
Suppose that $H^2_{\beta}$ and $H^2_{\beta'}$ are weighted Hardy spaces and
\[
w_{k+1}=\frac{\beta_{k+1}}{\beta_k}\geq \frac{\beta'_{k+1}}{\beta'_k}=w'_{k+1}, \ \ \ \text{for} \ k=0,1,2,\ldots.
\]
Let $\psi(z)$ be an analytic function on $\mathbb{D}$ with $\psi(\mathbb{D})\subseteq \mathbb{D}$ and $\psi(0)=0$. If $C_{\psi}$ is bounded on $H^2_{\beta}$, then $C_{\psi}$ is bounded on $H^2_{\beta'}$ and $\|C_{\psi}\|_{H^2_{\beta}}\geq\|C_{\psi}\|_{H^2_{\beta'}}$.
\end{theorem}

\begin{remark}
In C. C. Cowen's proof \cite{C90}, measure method was not used, while Hadamard product played an important role. Moreover, one can immediately generalize the above result a little. Let $\mathfrak{F}$ be a sequence of functions in $\psi\in \textrm{Hol}(\overline{\mathbb{D}})$. Suppose that the matrix representation of $X_{\mathfrak{F}}$ under the orthogonal base $\{z^n\}_{n=0}^{\infty}$ is lower triangular. Then the above conclusion also holds if we replace $C_{\psi}$ by $X_{\mathfrak{F}}$.
\end{remark}

\begin{proposition}\label{CPB}
Let $H^2_{\beta}$ be the weighted Hardy space induced by a weight sequence $w=\{w_k\}_{k=1}^{\infty}$ with $w_k\rightarrow 1$. Let $\psi\in \textrm{Hol}(\overline{\mathbb{D}})$ such that $\psi(\mathbb{D})\subseteq \mathbb{D}$.
Denote $\widetilde{\beta}=\{\widetilde{\beta}_n\}_{n=0}^{\infty}$, where $\widetilde{\beta}_n=(n+1)\beta_n$. If $C_{\psi}$ is bounded on $H^2_{\beta}$, then $C_{\psi}$ is bounded on $H^2_{\widetilde{\beta}}$. Moreover, if $\psi'(z)$ has no zero point on $\partial\mathbb{D}$, the converse is also true.
\end{proposition}

\begin{proof}
Define $D_w, D:H^2_{\beta}\rightarrow H^2_{\beta}$ by, for any $f(z)\in H^2_{\beta}$
\[
D_wf(z)=\sum\limits_{k=0}^{\infty}w_{k+1}\widehat{f}(k)z^k  \ \ \text{and} \ \
Df(z)=\sum\limits_{k=0}^{\infty}\frac{k+2}{k+1}\cdot \widehat{f}(k)z^k .
\]
We could write the two operators $D_w$ and $D$ in matrix form under the orthogonal base $\{z^k\}_{k=0}^{\infty}$ as follows,
\[
D_w=\begin{bmatrix}
w_1 & 0 & 0  & \cdots & 0 & \cdots \\
0   & w_2 & 0  & \cdots & 0 & \cdots \\
0   & 0   & w_3  & \cdots  & 0 & \cdots \\
\vdots   & \vdots & \vdots & \ddots &\vdots  &\vdots \\
0   & 0 & 0 & \cdots &w_k & \cdots \\
\vdots   & \vdots & \vdots&\vdots  &\vdots & \ddots
\end{bmatrix},    \
D=\begin{bmatrix}
2 & 0 & 0  & \cdots & 0 & \cdots \\
0   & \frac{3}{2} & 0  & \cdots & 0 & \cdots \\
0   & 0   & \frac{4}{3}  & \cdots  & 0 & \cdots \\
\vdots   & \vdots & \vdots & \ddots &\vdots  &\vdots \\
0   & 0 & 0 & \cdots &\frac{k+1}{k} & \cdots \\
\vdots   & \vdots & \vdots&\vdots  &\vdots & \ddots
\end{bmatrix}.
\]

Denote $\psi(0)=z_0\in\mathbb{D}$. We have
\begin{align*}
&\langle  \frac{{\psi}^{i+1}}{\widetilde{\beta_i}}, \frac{{\psi}^{j+1}}{\widetilde{\beta_j}}\rangle_{H^2_{\widetilde{\beta}}} \\
=& \ \frac{1}{\widetilde{\beta_i}\widetilde{\beta_j}}\cdot \sum\limits_{k=0}^{\infty} \overline{\widehat{{\psi}^{j+1}}(k)}\cdot \widehat{{\psi}^{i+1}}(k)\cdot \widetilde{\beta_k}^2 \\
=& \ \overline{\widehat{{\psi}^{j+1}}(0)} \widehat{{\psi}^{i+1}}(0)+\frac{1}{(i+1)(j+1){\beta}_i{\beta}_j}
\sum\limits_{k=1}^{\infty}(k+1) \overline{\widehat{{\psi}^{j+1}}(k)}(k+1)\widehat{{\psi}^{i+1}}(k){\beta}_k^2 \\
=& \ \overline{z_0}^{j+1}z_0^{i+1}+\frac{1}{(i+1)(j+1){\beta}_i{\beta}_j}
\sum\limits_{k=1}^{\infty}(\frac{k+1}{k}\cdot w_k)k\overline{\widehat{{\psi}^{j+1}}(k)} (\frac{k+1}{k}\cdot w_k)k\widehat{{\psi}^{i+1}}(k) {\beta}_{k-1}^2 \\
=& \ \overline{z_0}^{j+1}z_0^{i+1}+\frac{1}{(i+1)(j+1){\beta}_i{\beta}_j} \langle  DD_w(\psi^{i+1})', DD_w(\psi^{j+1})'\rangle_{H^2_{\beta}} \\
=& \ \overline{z_0}^{j+1}z_0^{i+1}+\frac{1}{{\beta}_i{\beta}_j} \langle  DD_wM_{\psi'}\psi^{i}, DD_wM_{\psi'}\psi^{j} \rangle_{H^2_{\beta}} \\
=& \ \overline{z_0}^{j+1}z_0^{i+1}+ \langle  DD_wM_{\psi'}(\frac{\psi^{i}}{\beta_i}), DD_wM_{\psi'}(\frac{\psi^{j}}{\beta_j}) \rangle_{H^2_{\beta}}.
\end{align*}
Then
\[
\Gamma_{H^2_{\widetilde{\beta}}}(\{\frac{\psi^{n+1}}{\widetilde{\beta}_n}\}_{n=0}^{\infty})=\Gamma_{H^2_{\beta}}(\{DD_wM_{\psi'}(\frac{\psi^{n}}{\beta_n})\}_{n=0}^{\infty})+
\begin{bmatrix}
\overline{z_0}\\
\overline{z_0}^{2}    \\
\overline{z_0}^{3}   \\
\vdots
\end{bmatrix}\begin{bmatrix}
{z_0} & {z_0}^{2} & {z_0}^{3}  & \cdots
\end{bmatrix}.
\]
Consequently,
\begin{align*}
&\|\Gamma_{H^2_{\beta}}(\{DD_wM_{\psi'}(\frac{\psi^{n}}{\beta_n})\}_{n=0}^{\infty})\|-\frac{|z_0|^2}{{1-|z_0|^2}} \\
\leq& \|\Gamma_{H^2_{\widetilde{\beta}}}(\{\frac{\psi^{n+1}}{\widetilde{\beta}_n}\}_{n=0}^{\infty})\| \\
\leq& \|\Gamma_{H^2_{\beta}}(\{DD_wM_{\psi'}(\frac{\psi^{n}}{\beta_n})\}_{n=0}^{\infty})\|+\frac{|z_0|^2}{{1-|z_0|^2}}.
\end{align*}
If $C_{\psi}$ is bounded on $H^2_{\beta}$, then $\Gamma_{H^2_{\beta}}(\{\frac{\psi^{n}}{\beta_n}\}_{n=0}^{\infty})$ is a bounded operator on $H^2$.
Since $\psi'(z)$ also belongs to $\textrm{Hol}(\overline{\mathbb{D}})$, $M_{\varphi'}$ is a bounded operator on $H^2_{\beta}$. Together with the boundedness of $D$ and $D_w$ on $H^2_{\beta}$, one can see that the boundedness of $\Gamma_{H^2_{\beta}}(\{\frac{\psi^{n}}{\beta_n}\}_{n=0}^{\infty})$ on $H^2$ implies the boundedness of
$\Gamma_{H^2_{\widetilde{\beta}}}(\{\frac{\psi^{n+1}}{\widetilde{\beta}_n}\}_{n=0}^{\infty})$ on $H^2$.

In addition, the boundedness of $\Gamma_{H^2_{\widetilde{\beta}}}(\{\frac{\psi^{n+1}}{\widetilde{\beta}_n}\}_{n=0}^{\infty})$ on $H^2$ is also equivalent to the boundedness of $\Gamma_{H^2_{\widetilde{\beta}}}(\{\frac{\psi^{n}}{\widetilde{\beta}_n}\}_{n=0}^{\infty})$ on $H^2$. Thus, if $C_{\psi}$ is bounded on $H^2_{\beta}$, then $C_{\psi}$ is bounded on $H^2_{\widetilde{\beta}}$.

Moreover, if $\psi'(z)$ has no zero point on $\partial\mathbb{D}$, then $M_{\varphi'}$ is lower bounded on $H^2_{\beta}$. Together with the lower boundedness of $D$ and $D_w$ on $H^2_{\beta}$, one can see that the boundedness of $\Gamma_{H^2_{\widetilde{\beta}}}(\{\frac{\psi^{n+1}}{\widetilde{\beta}_n}\}_{n=0}^{\infty})$ on $H^2$ implies the boundedness of $\Gamma_{H^2_{\beta}}(\{\frac{\psi^{n}}{\beta_n}\}_{n=0}^{\infty})$ on $H^2$. Consequently, if $C_{\psi}$ is bounded on $H^2_{\widetilde{\beta}}$, then $C_{\psi}$ is also bounded on $H^2_{\beta}$.
\end{proof}

Together with Theorem \ref{Controll} and Proposition \ref{CPB}, we can obtain the boundedness of the composition operator induced by arbitrary finite Blaschke product $B(z)$ with $B(0)=0$.

\begin{corollary}\label{boundcontr}
Let $H^2_{\beta}$ be the weighted Hardy space of polynomial growth induced by a weight sequence $w=\{w_k\}_{k=1}^{\infty}$.
Then, for any finite Blaschke product $B(z)$ with $B(0)=0$, $C_{B}$ is bounded on $H^2_{\beta}$.
\end{corollary}

\begin{proof}
Assume $\sup_k\{(k+1)|w_k-1|\}\leq M\in\mathbb{N}$. Let $H^2_{\widetilde{\beta}}$ be the weighted Hardy space induced by the weight sequence $\widetilde{w}=\{\widetilde{w_k}\}_{k=1}^{\infty}$, where $\widetilde{w_k}=\frac{k+M+1}{k+1}$. Let $H^2_{\widetilde{\beta}'}$ be the weighted Hardy space with $\widetilde{\beta_k}'=(k+1)^M$. Since
\[
\lim\limits_{k\rightarrow\infty}\frac{\widetilde{\beta_k}}{\widetilde{\beta_k}'}=\lim\limits_{k\rightarrow\infty}\frac{\prod_{j=0}^{k}\frac{j+M+1}{j+1}}{(k+1)^M}
=\lim\limits_{k\rightarrow\infty}\frac{\prod_{m=0}^{M}\frac{k+m+1}{m+1}}{(k+1)^M}=\prod\limits_{m=0}^{M}\frac{1}{m+1},
\]
$H^2_{\widetilde{\beta}}$ and $H^2_{\widetilde{\beta}'}$ are two equivalent weighted Hardy spaces.

As well known, $C_{B}$ is a bounded operator on the classical Hardy space $H^2$. Notice that $B'(z)$ has no zero point on $\partial\mathbb{D}$ (Theorem 2.1 in \cite{Car}), see also \cite{Fri}). Applying Proposition \ref{CPB} $M$ times, one can obtain that $C_{B}$ is bounded on $H^2_{\widetilde{\beta}'}$, and consequently $C_{B}$ is bounded on $H^2_{\widetilde{\beta}}$. Then, it follows from Theorem \ref{Controll} that $C_{B}$ is bounded on $H^2_{\beta}$.
\end{proof}

\section{Similar representation of analytic automorphisms and finite Blaschke products}

In this section, we discuss the similarity of the representation of analytic automorphisms and finite Blaschke products on a weighted Hardy space of polynomial growth.

\subsection{Similar representation on weighted Hardy spaces of polynomial growth} 

Following from the preliminaries in the previous section, we could give a proof of Theorem \ref{Mz} now.

{\bf Proof of Theorem \ref{Mz}.}
Without loss of generality, we may assume  $\varphi(z)=\frac{z_0-z}{1-\overline{z_0}z}$, $z_0\in\mathbb{D}\setminus\{0\}$. Let $B(z)=z\varphi(z)$. Denote $\mathfrak{F}=\{B^n\}_{n=0}^{\infty}$, $\mathfrak{F}_{\beta}=\{\frac{B^n}{\beta_n}\}_{n=0}^{\infty}$,
{\[
\widetilde{\mathfrak{F}}_1=\{B^n(z), zB^n(z); n=0,1,\ldots\} \ \ \ \text{and} \ \ \ \widetilde{\mathfrak{F}}_2=\{B^n(z), \varphi(z)B^n(z); n=0,1,\ldots\}.
\]}
It follows from Corollary \ref{twoBtotal} that $\widetilde{\mathfrak{F}}_1$ and $\widetilde{\mathfrak{F}}_2$ are total and finitely linear independent in $H^2_{\beta}$, respectively.

By Corollary \ref{boundcontr}, $C_{B}$ is bounded on $H^2_{\beta}$ and $H^2_{\beta^{-1}}$. Consequently, $D_{\beta}X_{\mathfrak{F}}D^{-1}_{\beta}$ and $D^{-1}_{\beta}X_{\mathfrak{F}}D_{\beta}$ are bounded on $H^2$.
Denote {
\[
\widetilde{\mathfrak{F}}_{1, \beta}=\{\frac{B^n(z)}{\beta_n}, \frac{zB^n(z)}{\beta_n}; n=0,1,\ldots\}, \ \ \    \widetilde{\mathfrak{F}}_{2, \beta}=\{\frac{B^n(z)}{\beta_n}, \frac{\varphi(z)B^n(z)}{\beta_n}; n=0,1,\ldots\}.
\]}
Then,
\[
\Gamma_{H^2_{\beta}}(\widetilde{\mathfrak{F}}_{1, \beta})=\langle  \widetilde{\mathfrak{F}}_{1, \beta}, \widetilde{\mathfrak{F}}_{1, \beta}\rangle_{H^2_{\beta}}=\begin{bmatrix}
\langle\mathfrak{F}_{\beta}, \mathfrak{F}_{\beta}\rangle_{H^2_{\beta}} &  \langle\mathfrak{F}_{\beta}, M_z\mathfrak{F}_{\beta}\rangle_{H^2_{\beta}}  \\
\langle M_z\mathfrak{F}_{\beta}, \mathfrak{F}_{\beta}\rangle_{H^2_{\beta}}   & \langle M_z\mathfrak{F}_{\beta}, M_z\mathfrak{F}_{\beta}\rangle_{H^2_{\beta}} \\
\end{bmatrix}
\]
where
\begin{align*}
&\langle\mathfrak{F}_{\beta}, \mathfrak{F}_{\beta}\rangle_{H^2_{\beta}}=(D^{-1}_{\beta}X^*_{\mathfrak{F}}D_{\beta})(D_{\beta}X_{\mathfrak{F}}D^{-1}_{\beta}), \\
&\langle\mathfrak{F}_{\beta}, M_z\mathfrak{F}_{\beta}\rangle_{H^2_{\beta}}=(D^{-1}_{\beta}X^*_{\mathfrak{F}}D_{\beta})(D^{-1}_{\beta}M_z^*D_{\beta})(D_{\beta}X_{\mathfrak{F}}D^{-1}_{\beta}),  \\
&\langle M_z\mathfrak{F}_{\beta}, \mathfrak{F}_{\beta}\rangle_{H^2_{\beta}}=(D^{-1}_{\beta}X^*_{\mathfrak{F}}D_{\beta})(D_{\beta}M_zD^{-1}_{\beta})(D_{\beta}X_{\mathfrak{F}}D^{-1}_{\beta}),   \\
&\langle M_z\mathfrak{F}_{\beta}, M_z\mathfrak{F}_{\beta}\rangle_{H^2_{\beta}}=
(D^{-1}_{\beta}X^*_{\mathfrak{F}}D_{\beta})(D^{-1}_{\beta}M_z^*D_{\beta})(D_{\beta}M_zD^{-1}_{\beta})(D_{\beta}X_{\mathfrak{F}}D^{-1}_{\beta}).
\end{align*}
Consequently, by the boundedness of $D_{\beta}X_{\mathfrak{F}}D^{-1}_{\beta}$ and $D_{\beta}M_zD^{-1}_{\beta}$ on $H^2$, $\Gamma_{H^2_{\beta}}(\widetilde{\mathfrak{F}}_{1, \beta})$ is bounded on $H^2$ and so is $D_{\beta}X_{\widetilde{\mathfrak{F}}_1}D^{-1}_{\beta}$. Similarly, we also obtain that $D^{-1}_{\beta}X_{\widetilde{\mathfrak{F}}_1}D_{\beta}$ is bounded on $H^2$.

By direct calculation, it is not difficult to see that

\[
\langle \frac{\sqrt{1-|z_0|^2}}{1-\overline{z_0}z}B^i(z), \frac{\sqrt{1-|z_0|^2}}{1-\overline{z_0}z}B^j(z)\rangle_{H^2}=\left\{\begin{array}{cc}
1, \ \ \ &\text{if} \ i=j \\
0, \ \ \ &\text{if} \ i\neq j
\end{array}\right.
\]
\[
\langle \frac{\sqrt{1-|z_0|^2}}{1-\overline{z_0}z}zB^i(z), \frac{\sqrt{1-|z_0|^2}}{1-\overline{z_0}z}zB^j(z)\rangle_{H^2} \\
=\left\{\begin{array}{cc}
1, \ \ \ &\text{if} \ i=j \\
0, \ \ \ &\text{if} \ i\neq j
\end{array}\right.
\]
and
\[
\langle \frac{\sqrt{1-|z_0|^2}}{1-\overline{z_0}z}zB^i(z), \frac{\sqrt{1-|z_0|^2}}{1-\overline{z_0}z}B^j(z)\rangle_{H^2} \\
=\left\{\begin{array}{cc}
z_0, \ \ \ &\text{if} \ i=j \\
0, \ \ \ &\text{if} \ i\neq j
\end{array}\right..
\]
Then, we have
\begin{align*}
& D_{\beta}X^*_{\widetilde{\mathfrak{F}}_1}M^*_{\frac{\sqrt{1-|z_0|^2}}{1-\overline{z_0}z}}M_{\frac{\sqrt{1-|z_0|^2}}{1-\overline{z_0}z}}X_{\widetilde{\mathfrak{F}}_1}D^{-1}_{\beta} \\
=& D_{\beta}\langle\frac{\sqrt{1-|z_0|^2}}{1-\overline{z_0}z}{\widetilde{\mathfrak{F}}_1}, \frac{\sqrt{1-|z_0|^2}}{1-\overline{z_0}z}{\widetilde{\mathfrak{F}}_1}\rangle_{H^2}D^{-1}_{\beta} \\
=& D_{\beta}\begin{bmatrix}
1 & \overline{z_0} & 0   & 0 & \cdots \\
z_0   & 1 & 0   & 0 & \cdots \\
0   & 0   & 1    & \overline{z_0} & \cdots \\
0   & 0 & z_0   & 1 & \cdots \\
\vdots   & \vdots &\vdots  &\vdots & \ddots
\end{bmatrix}D^{-1}_{\beta} \\
=& \begin{bmatrix}
1 & w_1^{-1}\overline{z_0} & 0   & 0 & \cdots \\
w_1z_0   & 1 & 0   & 0 & \cdots \\
0   & 0   & 1    & w_3^{-1}\overline{z_0} & \cdots \\
0   & 0 & w_3z_0   & 1 & \cdots \\
\vdots   & \vdots &\vdots  &\vdots & \ddots
\end{bmatrix}.
\end{align*}
Consequently,
\[
D_{\beta}X^*_{\widetilde{\mathfrak{F}}_1}M^*_{\frac{\sqrt{1-|z_0|^2}}{1-\overline{z_0}z}}M_{\frac{\sqrt{1-|z_0|^2}}{1-\overline{z_0}z}}X_{\widetilde{\mathfrak{F}}_1}D^{-1}_{\beta}
\]
is lower bounded on $l^2$. In addition,
\begin{align*}
&D_{\beta}X^*_{\widetilde{\mathfrak{F}}_1}M^*_{\frac{\sqrt{1-|z_0|^2}}{1-\overline{z_0}z}}M_{\frac{\sqrt{1-|z_0|^2}}{1-\overline{z_0}z}}X_{\widetilde{\mathfrak{F}}_1}D^{-1}_{\beta}\\
=& (D_{\beta}X^*_{\widetilde{\mathfrak{F}}_1}D^{-1}_{\beta}) (D_{\beta}M^*_{\frac{\sqrt{1-|z_0|^2}}{1-\overline{z_0}z}}M_{\frac{\sqrt{1-|z_0|^2}}{1-\overline{z_0}z}}D^{-1}_{\beta})
 (D_{\beta}X_{\widetilde{\mathfrak{F}}_1}D^{-1}_{\beta}),
\end{align*}
one can see that both of the first item
\[
D_{\beta}X^*_{\widetilde{\mathfrak{F}}_1}D^{-1}_{\beta}=(D^{-1}_{\beta}X_{\widetilde{\mathfrak{F}}_1}D_{\beta})^*
\]
and the second item
\[
D_{\beta}M^*_{\frac{\sqrt{1-|z_0|^2}}{1-\overline{z_0}z}}M_{\frac{\sqrt{1-|z_0|^2}}{1-\overline{z_0}z}}D^{-1}_{\beta}
\]
are bounded on $H^2$. Then, $D_{\beta}X_{\widetilde{\mathfrak{F}}_1}D^{-1}_{\beta}$ is lower bounded on $H^2$. Therefore, $\widetilde{\mathfrak{F}}_{1,\beta}$ is a Riesz base of $H^2_{\beta}$.

In the same way, we could obtain that $D_{\beta}X_{\widetilde{\mathfrak{F}}_2}D^{-1}_{\beta}$ is bounded and lower bounded on $H^2$, after a computation
\begin{align*}
& (D_{\beta}X^*_{\widetilde{\mathfrak{F}}_2}D^{-1}_{\beta})(D_{\beta}X_{\widetilde{\mathfrak{F}}_2}D^{-1}_{\beta}) \\
=& D_{\beta}X^*_{\widetilde{\mathfrak{F}}_2}X_{\widetilde{\mathfrak{F}}_2}D^{-1}_{\beta} \\
=& \begin{bmatrix}
1 & w_1^{-1}\overline{z_0} & 0   & 0 & \cdots \\
w_1z_0   & 1 & 0   & 0 & \cdots \\
0   & 0   & 1    & w_3^{-1}\overline{z_0} & \cdots \\
0   & 0 & w_3z_0   & 1 & \cdots \\
\vdots   & \vdots &\vdots  &\vdots & \ddots
\end{bmatrix}.
\end{align*}
Therefore, $\widetilde{\mathfrak{F}}_{2,\beta}$ is also a Riesz base of $H^2_{\beta}$.

It follows from $B(z)=z\varphi(z)$ and $\varphi(\varphi(z))=z$ that for every $n=0,1,2,\ldots$,
\[
C_{\varphi}(B^n(z))=B^n(\varphi(z))=B^n(z)
\]
and
\[
C_{\varphi}(zB^n(z))=\varphi(z)B^n(\varphi(z))=\varphi(z)B^n(z).
\]
Then, the composition operator $C_{\varphi}$ is just a base transformation operator from the Riesz base $\widetilde{\mathfrak{F}}_{1,\beta}$ to the Riesz base $\widetilde{\mathfrak{F}}_{2,\beta}$. Thus, $C_{\varphi}$ is a bounded invertible operator on $H^2_{\beta}$. Furthermore, $M_z\sim M_{\varphi}$, i.e., $M_z$ is weakly homogeneous on $H^2_{\beta}$.

By the way, $C_\varphi$ is bounded on $H^2_{\beta}$ if and only if the sequence $\{\frac{\varphi^n}{\beta_n}\}_{n=0}^{\infty}$ is a Riesz base in $H^2_{\beta}$. In particular, the boundedness of $C_\varphi$ on $H^2_{\beta}$ implies the sequence $\{\frac{\varphi^n}{\beta_n}\}_{n=0}^{\infty}$ is bounded (or quasinormed) in $H^2_{\beta}$. \qed

According to the boundedness of the composition operators induced by M\"{o}bius transformations, C. C. Cowen's theorem (Theorem \ref{Controll} in the present paper) could be extended as the following result.
\begin{corollary}\label{ECowen}
Let $H^2_{\beta}$ and $H^2_{\beta'}$ be two weighted Hardy spaces. Suppose that $H^2_{\beta}$ is of polynomial growth and
\[
w_{k+1}=\frac{\beta_{k+1}}{\beta_k}\geq \frac{\beta'_{k+1}}{\beta'_k}=w'_{k+1}, \ \ \ \text{for} \ k=0,1,2,\ldots.
\]
Let $\psi(z)$ be an analytic function on $\mathbb{D}$ with $\psi(\mathbb{D})\subseteq \mathbb{D}$. If $C_{\psi}$ is bounded on $H^2_{\beta}$, then $C_{\psi}$ is bounded on $H^2_{\beta'}$.
\end{corollary}
\begin{proof}
Let $z_0=\psi(0)\in \mathbb{D}$, and let $\varphi(z)=\frac{z_0-z}{1-\overline{z_0}z}$. Obviously, the polynomial growth of $H^2_{\beta}$ implies the polynomial growth of $H^2_{\beta'}$. Then, by Theorem \ref{Mz}, the composition operator $C_\varphi$ is bounded on $H^2_{\beta}$ and $H^2_{\beta'}$, respectively. Consider the function $ \varphi\circ \psi$. Since $C_{\varphi\circ \psi}$ is bounded on $H^2_{\beta}$ and $\varphi\circ \psi(0)=\psi(z_0)=0$, by C. C. Cowen's result Theorem \ref{Controll}, one can see that $C_{\varphi\circ \psi}$ is also bounded on $H^2_{\beta}$. Furthermore, it follows from the boundedness of $C_\varphi$ on $H^2_{\beta'}$ that $C_{\psi}=C_{\varphi\circ \psi}\circ C_{\varphi}$ is bounded on $H^2_{\beta'}$.
\end{proof}

Furthermore, we could obtain the boundedness of composition operators induced by analytic functions on $\overline{\mathbb{D}}$ on weighted Hardy spaces of polynomial growth.

{\bf Proof of Theorem \ref{CompBHol}.}
Assume $\sup_k\{(k+1)|w_k-1|\}\leq M\in\mathbb{N}$. Let $H^2_{\widetilde{\beta}}$ be the weighted Hardy space induced by the weight sequence $\widetilde{w}=\{\widetilde{w_k}\}_{k=1}^{\infty}$, where $\widetilde{w_k}=\frac{k+M+1}{k+1}$.

As well know, the composition operator $C_{\psi}$ is bounded on the classical Hardy space $H^2$, in fact, $\|C_{\psi}\|_{H^2}\leq 1$. As the same as the proof of Corollary \ref{boundcontr}, one can obtain that the composition operator $C_{\psi}$ is also bounded on $H^2_{\widetilde{\beta}}$ by Proposition \ref{CPB}. Then, it follows from Corollary \ref{ECowen} that $C_{\psi}$ is bounded on $H^2_{\beta}$. \qed

Next, let us consider finite Blaschke products.
\begin{lemma}\label{RBB}
Suppose that $H^2_{\beta}$ is a weighted Hardy space of polynomial growth. Let $B(z)=\prod^m_{j=1}\frac{z_{j}-z}{1-\overline{z_{j}}z}$ be a Blaschke product on $\mathbb{D}$ with $m$ distinct zero points. Denote $\mathfrak{F}=\{B^n\}_{n=0}^{\infty}$, $\mathfrak{F}_{\beta}=\{\frac{B^n}{\beta}\}_{n=0}^{\infty}$,
\[
\widetilde{\mathfrak{F}}=\bigcup\limits_{j=1}^{m}\frac{\mathfrak{F}}{1-\overline{z_{j}}z}=\{\frac{B^n}{1-\overline{z_{1}}z}, \cdots, \frac{B^n}{1-\overline{z_{m}}z}; n=0,1,2,\ldots\}.
\]
and
\[
\widetilde{\mathfrak{F}_{\beta}}=\bigcup\limits_{j=1}^{m}\frac{\mathfrak{F}_{\beta}}{1-\overline{z_{j}}z}=\{\frac{1}{1-\overline{z_{1}}z}\frac{B^n}{\beta_n}, \cdots, \frac{1}{1-\overline{z_{m}}z}\frac{B^n}{\beta_n}; n=0,1,2,\ldots\}.
\]
Then $\widetilde{\mathfrak{F}_{\beta}}$ is a Riesz base of $H^2_{\beta}$.
\end{lemma}

\begin{proof}
First of all, we may assume $z_1=0$ without loss of generality, because $C_{\varphi_{z_1}}$ is a bounded invertible operator on $H^2_{\beta}$ by Theorem \ref{Mz}, and consequently $B(\varphi_{z_1}(z))$ is also a Blaschke product on $\mathbb{D}$ with $m$ distinct zero points and $B(0)=0$.

Following from Theorem \ref{disBtotal}, $\widetilde{\mathfrak{F}_{\beta}}$ is a total and finitely linear independent sequence in $H^2_{\beta}$. By Lemma \ref{RBari}, it suffices to prove $D_{\beta}X_{\mathfrak{F}}D^{-1}_{\beta}$ is a bounded and lower bounded operator on $H^2$.

According to Corollary \ref{boundcontr}, $C_B$ is bounded on $H^2_{\beta}$ and $H^2_{\beta^{-1}}$. Consequently, both $D_{\beta}X_{\mathfrak{F}}D^{-1}_{\beta}$ and $D^{-1}_{\beta}X_{\mathfrak{F}}D_{\beta}$ are bounded on $H^2$.

On the other hand, since
\[
\langle \frac{B^i(z)}{1-\overline{z_i}z}, \frac{B^j(z)}{1-\overline{z_j}z} \rangle_{H^2}=\left\{\begin{array}{cc}
\frac{1}{1-\overline{z_i}z_j}, \ \ \ &\text{if} \ i=j \\
0, \ \ \ &\text{if} \ i\neq j
\end{array}\right.,
\]
we have
\begin{align*}
& (D_{\beta}X^*_{\widetilde{\mathfrak{F}}}D^{-1}_{\beta}) (D_{\beta}X_{\widetilde{\mathfrak{F}}}D^{-1}_{\beta}) \\
=& \begin{bmatrix}
D_0   & 0 & 0 & \cdots \\
0   & D_1 & 0 & \cdots \\
0   & 0   & D_2  & \cdots \\
\vdots   & \vdots  &\vdots & \ddots
\end{bmatrix}
\begin{bmatrix}
A   & 0 & 0 & \cdots \\
0   & A & 0 & \cdots \\
0   & 0   & A  & \cdots \\
\vdots   & \vdots  &\vdots & \ddots
\end{bmatrix}
\begin{bmatrix}
D_0^{-1}   & 0 & 0 & \cdots \\
0   & D_1^{-1} & 0 & \cdots \\
0   & 0   & D_2^{-1}  & \cdots \\
\vdots   & \vdots  &\vdots & \ddots
\end{bmatrix},
\end{align*}
where
\[
A=\begin{bmatrix}
\frac{1}{1-|z_1|^2} & \frac{1}{1-\overline{z_1}z_2} & \frac{1}{1-\overline{z_1}z_3}  & \cdots & \frac{1}{1-\overline{z_1}z_m} \\
\frac{1}{1-\overline{z_2}z_1}   & \frac{1}{1-|z_2|^2} & \frac{1}{1-\overline{z_2}z_3}  & \cdots & \frac{1}{1-\overline{z_2}z_m}  \\
\frac{1}{1-\overline{z_3}z_1}  & \frac{1}{1-\overline{z_3}z_2}   & \frac{1}{1-|z_3|^2}  & \cdots  & \frac{1}{1-\overline{z_3}z_m} \\
\vdots   & \vdots & \vdots & \ddots &\vdots \\
\frac{1}{1-\overline{z_m}z_1}   & \frac{1}{1-\overline{z_m}z_2} & \frac{1}{1-\overline{z_m}z_3} & \cdots &\frac{1}{1-|z_m|^2} \\
\end{bmatrix} \ \ \text{and} \ \ 
X^*_{\widetilde{\mathfrak{F}}}X_{\widetilde{\mathfrak{F}}}=\begin{bmatrix}
	A   & 0 & 0 & \cdots \\
	0   & A & 0 & \cdots \\
	0   & 0   & A  & \cdots \\
	\vdots   & \vdots  &\vdots & \ddots
\end{bmatrix},
\]
and for $k=0,1,2,\ldots$,
\[
D_k=\begin{bmatrix}
\frac{\beta_{km}}{\beta_{km}}   & 0 & 0  & \cdots &0 \\
0   & \frac{\beta_{km+1}}{\beta_{km}} & 0  &\cdots &0 \\
0   & 0   & \frac{\beta_{km+2}}{\beta_{km}}  & \cdots &0 \\
\vdots   & \vdots  &\vdots & \ddots &\vdots \\
0   & 0   &0 &\cdots & \frac{\beta_{km+m-1}}{\beta_{km}}
\end{bmatrix}.
\]
Then, $(D_{\beta}X^*_{\widetilde{\mathfrak{F}}}D^{-1}_{\beta}) (D_{\beta}X_{\widetilde{\mathfrak{F}}}D^{-1}_{\beta})$
is lower bounded on $H^2$. In addition to the boundedness of the first item
$D_{\beta}X^*_{\widetilde{\mathfrak{F}}}D^{-1}_{\beta}=(D^{-1}_{\beta}X_{\widetilde{\mathfrak{F}}}D_{\beta})^*$, we obtain that $D_{\beta}X_{\widetilde{\mathfrak{F}}}D^{-1}_{\beta}$ is lower bounded on $H^2$. This completes the proof.
\end{proof}

{\bf Proof of Theorem \ref{MB}.}
Without loss of generality, assume $B(z)=\prod^m_{j=1}\frac{z_{j}-z}{1-\overline{z_{j}}z}$.
Firstly, consider the case of that $B(z)$ has $m$ distinct zeros. Let
\[
\widetilde{\mathfrak{F}_{\beta}}=\bigcup\limits_{j=1}^{m}\frac{\mathfrak{F}_{\beta}}{1-\overline{z_{j}}z}=\{\frac{1}{1-\overline{z_{1}}z}\frac{B^n}{\beta_n}, \cdots, \frac{1}{1-\overline{z_{m}}z}\frac{B^n}{\beta_n}; n=0,1,2,\ldots\}.
\]
Then, by Lemma \ref{RBB}, $\widetilde{\mathfrak{F}_{\beta}}$ is a Riesz base of $H^2_{\beta}$.

For $j=1,2,\ldots,m$, denote by $\{e_{jn}=\frac{z^n}{\beta_n}\}_{n=0}^{\infty}$ the orthonormal base of the $j$-th space of $\bigoplus_{1}^{m}H^2_{\beta}$. Then, the sequence
\[
\mathfrak{E}_{\beta}=\{e_{jn}; j=1,2,\ldots,m \ \text{and} \ n=0,1,2,\ldots\}
\]
is an orthonormal base of $\bigoplus_{1}^{m}H^2_{\beta}$.

Define the linear operator $X: \bigoplus_{1}^{m}H^2_{\beta}\rightarrow H^2_{\beta}$ by
\[
X(e_{jn})=\frac{1}{1-\overline{z_j}z}\frac{B^n}{\beta_n}.
\]
Then $X$ maps the orthonormal base $\mathfrak{E}_{\beta}$ to the Riesz base $\widetilde{\mathfrak{F}_{\beta}}$, and consequently $X$ is a bounded invertible operator. Obviously,
\[
X (\bigoplus\limits_{1}^{m}M_z)= M_B X.
\]
Thus, $M_B\sim \bigoplus_{1}^{m}M_z$.

Now consider the general case. For any Blaschke product $B(z)$ with order $m$, there exists an analytic automorphism $\varphi_{z_0}(z)=\frac{z_0-z}{1-\overline{z_0}z}$ such that, $\varphi_{z_0}(B(z))$ is a Blaschke product with $m$ distinct zeros.

Following from $M_{\varphi_{z_0}(B)}\sim \bigoplus_{1}^{m}M_z$ and Theorem \ref{Mz}, we have
\[
M_B=\varphi_{z_0}(M_{\varphi_{z_0}(B)})\sim \varphi_{z_0}(\bigoplus\limits_{1}^{m}M_z)=\bigoplus\limits_{1}^{m}M_{\varphi_{z_0}}\sim\bigoplus\limits_{1}^{m}M_z.
\]
This completes the proof. \qed

\subsection{Examples and necessity of polynomial growth condition}

The weighted Hardy spaces of polynomial growth cover the weighted Bergman
spaces, the weighted Dirichlet spaces and many weighted Hardy spaces defined without measures.

\begin{example}\label{BergmanDS}
Let $dA$ denote the Lebesgue area measure on the unit open disk $\mathbb{D}$, normalized so that the measure of $\mathbb{D}$ equals $1$. For $\alpha\geq0$, the weighted Bergman space $A_{\alpha}^2$ is the space of analytic functions on $\mathbb{D}$ which are square-integrable with respect to the measure $dA_{\alpha}(z)=(\alpha+1)(1-|z|^2)^{\alpha}dA(z)$. As well known, it could be seemed as a weighted Hardy space $H^2_{\beta^{(\alpha)}}$, where $\beta^{(\alpha)}_0=1$ and for $k=1,2,\ldots$, $\beta_k^{(\alpha)}=\prod\limits_{j=1}^{k}w_{j}^{(\alpha)}$ with $w_{j}^{(\alpha)}=\sqrt{\frac{j+1}{j+2\alpha+1}}$. Since
\[
\lim\limits_{j\rightarrow\infty}(j+1)|w^{(\alpha)}_j-1|=\lim\limits_{j\rightarrow\infty}(j+1)\cdot\frac{1-\frac{j+1}{j+2\alpha+1}}{1+\sqrt{\frac{j+1}{j+2\alpha+1}}}=\alpha,
\]
every weighted Bergman space $A_{\alpha}^2$ satisfies the polynomial growth condition.

Similarly, the weighted Dirichlet space $D_{\lambda}^2$, for $\lambda\geq0$, could be seemed as a weighted Hardy space $H^2_{\beta^{(\lambda)}}$, where $\beta^{(\lambda)}_0=1$ and for $k=1,2,\ldots$, $\beta_k^{(\lambda)}=\prod_{j=1}^{k}w_{j}^{(\lambda)}$ with $w_{j}^{(\lambda)}=\sqrt{\frac{j+2\lambda+1}{j+1}}$; the Sobolev space $R(\mathbb{D})$ could be seemed as a weighted Hardy space $H^2_{\beta^{(S)}}$, where  $\beta_k^{(S)}=\prod_{j=1}^{k}w_{j}^{(S)}=\sqrt{\frac{k+1}{(3k^4-k^2+2k+1)\pi}}$ for $k=0,1,2,\ldots$. Both the weighted Dirichlet space and Sobolev space satisfy the polynomial growth condition. In particular, $H^2_{\frac{1}{\beta^{(\lambda)}}}$ is just the weighted Bergman space $A_{\lambda}^2$.
\end{example}

Beyond the weighted Hardy spaces defined by measures such as the weighted Bergman spaces, there are many weighted Hardy spaces of polynomial growth defined without measures.

\begin{example}
Let $H^2_{\beta}$ be the weighted Hardy space with
\[
\beta_0=1 \ \ \ \text{and} \ \ \  \beta_k=\frac{1}{\ln(k+1)}, \ \  \  \ \text{for} \ k=1,2,\ldots.
\]
Then $H^2_{\beta}$ satisfies the polynomial growth condition. This weighted Hardy space is not equivalent to the weighted Bergman
space, the weighted Dirichlet space or the Sobolev space.
\end{example}

In addition, the polynomial growth condition does not require the monotonicity of weight sequence.

\begin{example}
Given $\alpha\geq0$. Let $H^2_{\widetilde{\beta}}$ be the weighted Hardy space induced by the weight sequence $\widetilde{w}=\{\widetilde{w_k}\}_{k=1}^{\infty}$, where $\widetilde{w_k}$ is choosen to be the Bergman weight $w_{j}^{(\alpha)}$ or its reciprocal randomly, i.e.,
\[
\widetilde{w_k}=\left(\sqrt{\frac{j+1}{j+2\alpha+1}}\right)^{\epsilon_j}, \ \ \ \epsilon_j\in\{-1,1\}, \ \ \ \ \text{for} \ k=1,2,\ldots.
\]
Then $H^2_{\widetilde{\beta}}$ satisfies the polynomial growth condition, although $\widetilde{w}=\{\widetilde{w_k}\}_{k=1}^{\infty}$ is not a monotone sequence.
\end{example}

Now we will illustrate the necessity of the setting of polynomial growth condition. It is given a counterexample to show that $M_z$ could be not similar to $M_{\varphi}$, for some $\varphi\in \textrm{Aut}(\mathbb{D})$, on a weighted Hardy space of intermediate growth. In fact, we will show that the sequence $\{\frac{\varphi^n}{\beta_n}\}_{n=0}^{\infty}$ is not bounded in a weighted Hardy space of intermediate growth.

Let us begin from some estimations of the norms of some functions and composition operators on the classical Hardy space and the weighted Bergman spaces.

\begin{lemma}\label{dnlower}
Let $\varphi_t(z)=\frac{t-z}{1-tz}$, for any $t\in (0,1)$. Then the norm of $\alpha-$th power of the derivative of $\varphi_t(z)$ in the classical Hardy space $H^2$ tends to infinite as the positive number $\alpha$ tends to infinite. More precisely, for any $\alpha\geq \frac{1}{2}$,
\[
\|(\varphi_t')^{\alpha}(z)\|_{H^2} \geq \frac{(1+t)^{\alpha}}{(1-t)^{\alpha-1}\sqrt{2\pi(2\alpha-1)}}.
\]
\end{lemma}
\begin{proof}
Notice that
\begin{align*}
\|(\varphi_t')^{\alpha}(z)\|^2_{H^2}
&=\frac{(1-t^2)^{2\alpha}}{2\pi}\int_{0}^{2\pi}\frac{1}{(1-t{\textrm e}^{{\textbf i}\theta})^{2\alpha}(1-t{\textrm e}^{-{\textbf i}\theta})^{2\alpha}} d\theta \\
&=\frac{(1-t^2)^{2\alpha}}{2\pi}\int_{0}^{2\pi}\frac{1}{(1+t^2-2t\cos\theta)^{2\alpha}} d\theta \\
&\geq\frac{(1-t^2)^{2\alpha}}{2\pi}\cdot 2\cdot \int_{0}^{\pi}\frac{\sin\theta}{(1+t^2-2t\cos\theta)^{2\alpha}} d\theta \\
&=\frac{(1-t^2)^{2\alpha}}{2\pi}\cdot2 \cdot\frac{1}{2t(2\alpha-1)}\cdot(\frac{1}{(1-t)^{2(2\alpha-1)}}-\frac{1}{(1+t)^{2(2\alpha-1)}}) \\
&\geq\frac{(1-t^2)^{2\alpha}}{2\pi}\cdot\frac{1}{t(2\alpha-1)}\cdot\frac{t}{(1-t)^{2(2\alpha-1)}} \\
&=\frac{(1+t)^{2\alpha}}{2\pi(2\alpha-1)(1-t)^{2\alpha-2}}.
\end{align*}
Then,
\[
\|(\varphi_t')^{\alpha}(z)\|_{H^2} \geq \frac{(1+t)^{\alpha}}{(1-t)^{\alpha-1}\sqrt{2\pi(2\alpha-1)}}\rightarrow\infty \ \ \ \text{as} \ \alpha\rightarrow\infty.
\]
\end{proof}

\begin{lemma}\label{Blower}
Let $\varphi_t(z)=\frac{t-z}{1-tz}$, for any $t\in (0,1)$. Then for any $\alpha\geq \frac{1}{2}$, if $n$ is large enough, we have
\[
\frac{\|\varphi^n_t(z)\|_{A_{\alpha}^2}}{\|z^n\|_{A_{\alpha}^2}}\geq \frac{1}{2}\|(\varphi_t')^{\alpha}(z)\|_{H^2}.
\]
\end{lemma}

\begin{proof}
Following from \cite{KMis12}, one can see the unitary representation $D_{\alpha}^+$ for the analytic automorphism group $\textrm{Aut}(\mathbb{D})$ on the weighted Bergman space $A_{\alpha}^2$,
\[
D_{\alpha}^+(g^{-1})(f)=(g')^{\alpha}\cdot(f\circ g), \ \ f\in A_{\alpha}^2, \ g\in \widetilde{\textrm{Aut}(\mathbb{D})},
\]
where $\widetilde{\textrm{Aut}(\mathbb{D})}$ is the universal covering group of $\textrm{Aut}(\mathbb{D})$.

Then, we have
\[
\|\varphi^n_t(z)\|_{A_{\alpha}^2}=\|(\varphi_t'(z))^{\alpha}\cdot(\varphi^n_t\circ \varphi_t(z))\|_{A_{\alpha}^2}=\|z^n(\varphi_t'(z))^{\alpha}\|_{A_{\alpha}^2}.
\]
Denote $(\varphi_t'(z))^{\alpha}=\sum_{k=0}^{\infty}c_kz^k$. There exists $N\in\mathbb{N}$ such that
\[
\sum\limits_{k=0}^{N}|c_k|^2\geq\frac{1}{2}\sum\limits_{k=0}^{\infty}|c_k|^2=\frac{1}{2}\|(\varphi_t')^{\alpha}(z)\|^2_{H^2}.
\]
Furthermore, if $n$ is large enough, we have
\[
\frac{\beta_{k+n}^{(\alpha)}}{\beta_{n}^{(\alpha)}}\geq \frac{\sqrt{2}}{2}, \ \ \ \text{for} \ k=1,2,\ldots,N.
\]
Then,
\[
\frac{\|\varphi^n_t(z)\|_{A_{\alpha}^2}}{\|z^n\|_{A_{\alpha}^2}}=\frac{\|z^n(\varphi_t'(z))^{\alpha}\|_{A_{\alpha}^2}}{\beta_n^{(\alpha)}} 
\geq \sqrt{\sum\limits_{k=0}^{N}|c_k|^2\left(\frac{\beta_{k+n}^{(\alpha)}}{\beta_{n}^{(\alpha)}}\right)^2} 
\geq \frac{1}{2}\|(\varphi_t')^{\alpha}\|_{H^2}.
\]
\end{proof}

Next, we could show that $C_{\varphi_t}$ is unbounded on a class of weighted Hardy spaces of intermediate growth.

{\bf Proof of Theorem \ref{img}.}
Since
\begin{align*}
\frac{\|\varphi^{n_j}_t(z)\|^2_{\beta^{-1}}}{\|z^{n_j}\|^2_{\beta^{-1}}}&=\sum\limits_{k=0}^{\infty}|\widehat{\varphi^{n_j}_t}(k)|^2\cdot \frac{\beta_{n_j}^2}{\beta_k^2}  \\
&\geq  \sum\limits_{k=0}^{n_j}|\widehat{\varphi^{n_j}_t}(k)|^2\cdot \left(\frac{\beta_k^{(\alpha_j)}}{\beta_{n_j}^{(\alpha_j)}}\right)^2  \\
&\geq  \sum\limits_{k=0}^{\infty}|\widehat{\varphi^{n_j}_t}(k)|^2\cdot \left(\frac{\beta_k^{(\alpha_j)}}{\beta_{n_j}^{(\alpha_j)}}\right)^2 - \sum\limits_{k=n_j+1}^{\infty}|\widehat{\varphi^{n_j}_t}(k)|^2 \\
&\geq \frac{\|\varphi^{n_j}_t(z)\|^2_{A_{\alpha_j}^2}}{\|z^{n_j}\|^2_{A_{\alpha_j}^2}}-1,
\end{align*}
it follows from Lemma \ref{dnlower} and Lemma \ref{Blower} that
\[
\frac{\|\varphi^{n_j}_t(z)\|^2_{\beta^{-1}}}{\|z^{n_j}\|^2_{\beta^{-1}}} \rightarrow\infty, \ \ \ \text{as} \ j\rightarrow\infty.
\]
Then the composition operator $C_{\varphi_t}: H^2_{\beta^{-1}}\rightarrow H^2_{\beta^{-1}}$ is unbounded, and so is $C_{\varphi_t}: H^2_{\beta}\rightarrow H^2_{\beta}$ by Theorem \ref{CorA}.

Notice that $\lim_{k\rightarrow\infty}(k+1)(w_k-1)=-\infty$ if and only if $\lim_{k\rightarrow\infty}(k+1)(\frac{1}{w_k}-1)=+\infty$. By Proposition \ref{CorA},
it suffices to consider the case of $\lim_{k\rightarrow\infty}(k+1)(w_k-1)=+\infty$. Let $\alpha_j=j$ for any $j\in \mathbb{N}$. Then for any $j\in \mathbb{N}$,
there exists a positive integer $m_j$ such that for every $k>m_j$, $w_k\geq \frac{1}{w_k^{(\alpha_j)}}$. Since
\[
\lim\limits_{k\rightarrow\infty} \beta_k\cdot \beta_k^{(\alpha_j)}=\infty,
\]
there exists a positive integer $n_j>m_j$, such that $\beta_{n_j}\geq \frac{\beta_{m_j}}{\beta_{n_j}^{(\alpha_j)}}$. Then, for every $0\leq k\leq m_j$, we have
\[
\frac{\beta_{n_j}}{\beta_k}\geq \frac{\beta_{m_j}}{\beta_k\beta_{n_j}^{(\alpha_j)}}\geq \frac{1}{\beta_{n_j}^{(\alpha_j)}}  \geq\frac{\beta_k^{(\alpha_j)}}{\beta_{n_j}^{(\alpha_j)}}.
\]
and for every $m_j< k\leq n_j$, we have
\[
\frac{\beta_{n_j}}{\beta_k}= \prod\limits_{i=k+1}^{n_j}w_{i}\geq \prod\limits_{i=k+1}^{n_j}\frac{1}{w_i^{(\alpha_j)}} =\frac{\beta_k^{(\alpha_j)}}{\beta_{n_j}^{(\alpha_j)}}.
\]
This completes the proof. \qed

\begin{example}\label{nln}
Let $\beta=\{\beta_k\}_{k=0}^{\infty}$, where $\beta_0=1$ and $\beta_k=(k+1)^{\ln (k+1)}$ for $k\geq 1$. Obviously, $w_k={\textrm e}^{\ln^2(k+1)-\ln^2 k}$ for all $k=1,2,\ldots$. It is not difficult to see that
$w_k={\textrm e}^{\ln^2(k+1)-\ln^2 k}\rightarrow 1$ as $k\rightarrow \infty$ and $\lim\limits_{k\rightarrow\infty}(k+1)(w_k-1)=+\infty$.
Then, by Theorem \ref{img}, the composition operator $C_{\varphi_t}: H^2_{\beta^{-1}}\rightarrow H^2_{\beta^{-1}}$ is unbounded, and so is $C_{\varphi_t}: H^2_{\beta}\rightarrow H^2_{\beta}$ by Proposition \ref{CorA}.
\end{example}

\begin{remark}
In the above example, $C_{\varphi_t}$ is unbounded on $H^2_{\beta^{-1}}$. However, it follows from C. C. Cowen's result (Theorem \ref{Controll} in the present article) that $C_{z\varphi_t}$ is bounded on $H^2_{\beta^{-1}}$. This means that the boundedness of $C_{z\varphi_t}$ does not imply the boundedness of $C_{\varphi_t}$.
\end{remark}

\section{Jordan decomposition and similarity classification of the representation of analytic functions on weighted Hardy spaces of polynomial growth}

In this section, we always assume that $H^2_{\beta}$ is the weighted Hardy space of polynomial growth induced by a weight sequence $w=\{w_k\}_{k=1}^{\infty}$.
We aim to study the Jordan decomposition of the representation of $\textrm{Hol}(\overline{\mathbb{D}})$ on the weighted Hardy space $H^2_{\beta}$. Strongly irreducible operator is a suitable substitute for Jordan block, and strongly irreducible decomposition is seemed as the Jordan decomposition of an infinite-dimensional matrix.

Let us recall some relevant basic concepts and notations.
Let $\mathcal{H}$ be a complex separable Hilbert space and $\mathcal{L}(\mathcal{H})$ be the set of all bounded linear operators from $\mathcal{H}$ to itself. For any $T\in \mathcal{L}(\mathcal{H})$, denote by $\{T\}'_{\mathcal{H}}$ the commutant of $T$ in $\mathcal{L}(\mathcal{H})$. Moreover, a nonzero idempotent $P$ in the commutant $\{T\}'_{\mathcal{H}}$ is said to be minimal if for each idempotent $Q\in \{T\}'_{\mathcal{H}}$, ${\textrm Ran}Q \subsetneqq {\textrm Ran}P$ implies $Q=0$. Here
${\textrm Ran} P$ denotes the range of the operator $P$. A bounded operator $T$ on a Hilbert space $\mathcal{H}$ is said to be strongly irreducible if there is no
nontrivial idempotent operator in its commutant. If $\mathfrak{P}=\{P_j; j=1,2,\ldots m\}$ be a family of minimal idempotents such that
\[
\sum\limits_{j=1}^{\textit{l}}P_j={\textbf I}, \ \ \ \text{and} \ \ \ P_iP_j=0 \ \  \text{if} \  i\neq j.
\]
We say that $\mathfrak{P}$ is a strongly irreducible decomposition of $T$.
For convenience, we use $\mathcal{H}^{(m)}$ to denote the direct sum of
$m$ copies of $\mathcal{H}$ and $T^{(m)}$ to denote the direct sum of $m$ copies of $T$ acting on $\mathcal{H}^{(m)}$.

Firstly, we need the following theorem. The
version of the following theorem on the Hardy space was obtained in \cite{Th77}. For more general
results on the Hardy space, see \cite{Th76} and \cite{C78}. The method in \cite{Th77} also works on the weighted
Bergman spaces by replacing the reproducing kernel on the Hardy space by one on the weighted
Bergman spaces $A_{\alpha}^2$ \cite{JZ}. Notice that
\[
k(z,\omega)=\sum\limits_{k=0}^{\infty}\frac{\overline{\omega}^kz^k}{\beta^2_k}
\]
is the reproducing kernel of the weighted Hardy space $H^2_{\beta}$. Similarly, we can give a proof of the following theorem as the same as the proof
in \cite{Th77} (pp. 524-528), except replacing the reproducing kernel on the classical Hardy space by one on the weighted Hardy space $H^2_{\beta}$.

\begin{theorem}\label{fhB}
For any $f\in \textrm{Hol}(\overline{\mathbb{D}})$, there exist a finite Blaschke
product $B$ and a function $h\in \textrm{Hol}(\overline{\mathbb{D}})$ such that
$f=h\circ B$ and $\{M_f \}_{H^2_{\beta}}' = \{M_B \}_{H^2_{\beta}}'$.
\end{theorem}

Furthermore, we could obtain an analogue of Jordan decomposition for the representation of $\textrm{Hol}(\overline{\mathbb{D}})$ on a weighted Hardy space $H^2_{\beta}$ of polynomial growth. Notice that the following four facts hold.
\begin{enumerate}
  \item [(1)] $\{M_z \}_{H^2_{\beta}}' =\{M_f; f\in H^{\infty}_{\beta}\}$ and consequently $M_z$ is strongly irreducible.
  \item [(2)] For any $m\in\mathbb{N}$, $\{\bigoplus_{j=1}^m M_z \}_{\bigoplus_{j=1}^m H^2_{\beta}}' =\{M_F; F\in {\textrm M_m}(H^{\infty}_{\beta})\}$.
  \item [(3)] Let $P_j$ be the projection from $\bigoplus_{j=1}^m H^2_{\beta}$ onto the $j$-th component. Then $\mathfrak{P}=\{P_j; j=1,2,\ldots m\}$ is a strongly irreducible decomposition of $\bigoplus_{j=1}^m M_z$.
  \item [(4)] Let $B$ be a finite Blaschke product with order $m$. Then $M_B\sim\bigoplus_{j=1}^m M_z$ by theorem \ref{MB}, i.e., there exists a bounded invertible operator
      \[
      X:\bigoplus\limits_{j=1}^m H^2_{\beta}\rightarrow H^2_{\beta}
      \]
      such that
      \[
      X(\bigoplus\limits_{j=1}^m M_z) X^{-1}=M_B.
      \]
      Then $\{XP_jX^{-1}; j=1,2,\ldots m\}$ is also a strongly irreducible decomposition of $M_B$, where $\mathfrak{P}=\{P_j; j=1,2,\ldots m\}$ is defined in above $(3)$.
\end{enumerate}
Then the method of Jiang and Zheng, to prove the following result on weighted
Bergman spaces $A_{\alpha}^2$ (Lemma 3.3 in \cite{JZ}), also works on the weighted space $H^2_{\beta}$.

\begin{theorem}\label{JDT}
Given any $f\in \textrm{Hol}(\overline{\mathbb{D}})$. Suppose that $f=h\circ B$, where $h\in \textrm{Hol}(\overline{\mathbb{D}})$ and $B$ is a finite Blaschke
product with order $m$ such that $\{M_f \}_{H^2_{\beta}}' = \{M_B \}_{H^2_{\beta}}'$. Then
\[
M_f\sim\bigoplus\limits_1^m M_h,
\]
$\{M_h \}_{H^2_{\beta}}' = \{M_z \}_{H^2_{\beta}}'$ and $M_h$ is strongly irreducible.
\end{theorem}

Now we have obtained the Jordan decomposition of the representation of $\textrm{Hol}(\overline{\mathbb{D}})$ on $H^2_{\beta}$. Since there is no standard form of strongly irreducible operators in the sense of similarity, a natural question is how to characterize the similarity of Jordan representation. First, let us consider the "Jordan block" $M_h$ with $h\in \textrm{Hol}(\overline{\mathbb{D}})$ and $\{M_{h} \}_{H^2_{\beta}}' = \{M_z \}_{H^2_{\beta}}'$.

\begin{lemma}\label{sisim}
Let $h_1, h_2\in \textrm{Hol}(\overline{\mathbb{D}})$ with $\{M_{h_1} \}_{H^2_{\beta}}' = \{M_{h_2} \}_{H^2_{\beta}}' = \{M_z \}_{H^2_{\beta}}'$. Then
\[
M_{h_1}\sim M_{h_2}
\]
if and only if there exists a M\"{o}bius transformation $\varphi$ such that
\[
h_2=h_1\circ \varphi.
\]
\end{lemma}

\begin{proof}
If there is a M\"{o}bius transformation $\varphi$ such that $h_2=h_1\circ \varphi$, then
\[
C_{\varphi} M_{h_1}= M_{h_2} C_{\varphi},
\]
where the composition operator $C_{\varphi}$ is a bounded invertible operator on $H^2_{\beta}$ by Theorem \ref{Mz}.

Suppose that $X$ is a bounded invertible operator on $H^2_{\beta}$ such that
\[
X M_{h_1}= M_{h_2} X.
\]
Then
\[
X \{M_{h_1} \}_{H^2_{\beta}}'  X^{-1}= \{M_{h_2} \}_{H^2_{\beta}}'.
\]
Since $\{M_{h_1} \}_{H^2_{\beta}}' = \{M_{h_2} \}_{H^2_{\beta}}' = \{M_z \}_{H^2_{\beta}}'=H^{\infty}_{\beta}$,
there exists a function $g\in H^{\infty}_{\beta}$ such that
\[
X M_{z} X^{-1}= M_{g}.
\]
This implies that $X$ is just the composition operator $C_{g}$. Similarly, one can obtain that $X^{-1}$ is also a composition operator $C_{\psi}$, where $\psi$ is a function in $H^{\infty}_{\beta}$. Notice that
\[
g(\mathbb{D})\subseteq\mathbb{D}, \ \psi(\mathbb{D})\subseteq\mathbb{D}, \ \ \text{and} \ \ g\circ\psi=\psi\circ g={\textrm Id}_{\mathbb{D}}.
\]
Then $g$ is an analytic automorphism on $\mathbb{D}$. Furthermore, it follows from
\[
C_{g} M_{h_1}= M_{h_2} C_{g}
\]
that
\[
h_2=h_1\circ g.
\]
The proof is finished.
\end{proof}

\begin{remark}\label{CMz}
As a special case of above lemma, the converse of Theorem \ref{Mz} is also true. Let $\psi\in \textrm{Hol}({\mathbb{D}})$. If $M_z\sim M_{\psi}$ on $H^2_{\beta}$, then $\psi$ must be a  M\"{o}bius transformation. In another word,
if $C_{\psi}$ is a bounded invertible operator on $H^2_{\beta}$, then $\psi$ is a  M\"{o}bius transformation.
\end{remark}

To study the similarity of the operator representation, K-theory of Banach algebra is a powerful technique, for instance, a similarity classification of Cowen-Douglas operators was given by using the ordered K-group of the commutant algebra as an invariant \cite{J04} and \cite{JGJ}.
Let $\mathcal{B}$ be a Banach algebra and $Proj(\mathcal{B})$ be the set of all idempotents in $\mathcal{B}$. The algebraic equivalence "$\sim_{a}$" is introduced in $Proj(\mathcal{B})$. Let $e$ and $\widetilde{e}$ be two elements in $Proj(\mathcal{B})$. We say that $e\sim_{a} \widetilde{e}$ if there are two elements $x,y \in\mathcal{B}$ such that
\[
xy=e \ \ \ \text{and} \ \ \ yx=\widetilde{e}.
\]
Let ${\textbf Proj}(\mathcal{B})$ denote the algebraic equivalence classes of $Proj(\mathcal{B})$ under algebraic equivalence "$\sim_{a}$". Let
\[
{\textrm M}_{\infty}(\mathcal{B}) = \bigcup\limits_{n=1}^{\infty} {\textrm M}_{n}(\mathcal{B}),
\]
where ${\textrm M}_{n}(\mathcal{B})$ is the algebra of $n\times n$ matrices with entries in $\mathcal{B}$. Set
\[
\bigvee(\mathcal{B})={\textbf Proj}({\textrm M}_{\infty}(\mathcal{B})).
\]
Then $\bigvee({\textrm M}_{n}(\mathcal{B}))$ is isomorphic to $\bigvee(\mathcal{B})$.
The direct sum of two matrices gives a natural addition in ${\textrm M}_{\infty}(\mathcal{B})$ and hence induces an addition "$+$" in ${\textbf Proj}({\textrm M}_{\infty}(\mathcal{B}))$ by
\[
[p] + [q] = [p  \oplus q],
\]
where $[p]$ denotes the equivalence class of the idempotent $p$. Furthermore, $(\bigvee(\mathcal{B}), +)$ forms a semigroup and depends on $\mathcal{B}$
only up to stable isomorphism, and then $K_0(\mathcal{B})$ is the Grothendieck group of $\bigvee(\mathcal{B})$.

\begin{theorem}[\cite{CFJ}, see also in \cite{JW}]\label{CFJ-K}
Let $T$ be a bounded operator on a Hilbert space $\mathcal{H}$. The following are equivalent:
\begin{enumerate}
\item [(1)] $T$ is similar to $\sum_{i=1}^k \oplus A_i^{(n_i)}$ under the space decomposition $\mathcal{H}=\sum_{i=1}^k \oplus \mathcal{H}_i^{(n_i)}$,
where $k$ and $n_i$ are finite, $A_i$ is strongly irreducible and $A_i$ is not similar to $A_j$ if $i\neq j$.
Moreover, $T^{(n)}$ has a unique strongly irreducible decomposition up to similarity.
\item [(2)] The semigroup $\bigvee (\{T\}')$ is isomorphic to the semigroup $\mathbb{N}^{(k)}$, where $\mathbb{N}$ is the set of all natural numbers $\{0, 1, 2,\ldots\}$ and the isomorphism $\phi$ sends
\[
[I]\rightarrow n_1e_1 + n_2e_2 +\cdots+n_ke_k,
\]
where $\{e_i\}_{i=1}^k$ are the generators of $\mathbb{N}^{(k)}$ and $n_i\neq 0$.
\end{enumerate}
\end{theorem}

\begin{corollary}[\cite{CFJ}, see also in \cite{JW}]\label{n1}
Let $T_1$ and $T_2$ be two strongly irreducible operators on a Hilbert space $\mathcal{H}$, and $T\sim T_1^{(n_1)}\oplus T_2^{(n_2)}$.
If $\bigvee (\{T\}')$ is isomorphic to the semigroup $\mathbb{N}$, then $T_1$ is similar to $T_2$.
\end{corollary}

To study the similarity of the Jordan representation, we would also use some fundamental properties of Fredholm operators (refer to \cite{Dou}). A bounded operator $T$ on a Hilbert space $\mathcal{H}$ is said to be Fredholm, if ${\textrm dim} {\textrm Ker} T$ and ${\textrm dim} {\textrm Ker} T^*$
are finite. Moreover, the Fredholm index of $T$ is defined by
\[
{\textrm ind} T={\textrm dim} {\textrm Ker} T-{\textrm dim} {\textrm Ker} T^*.
\]
Notice that Fredholm index is a similar invariant and for any $n\in\mathbb{N}$, $T^{(n)}$ is also Fredholm with ${\textrm ind} T^{(n)}=n\cdot{\textrm ind} T$.

\begin{theorem}\label{nh}
Let $h_1, h_2\in \textrm{Hol}(\overline{\mathbb{D}})$ with $\{M_{h_1} \}_{H^2_{\beta}}' = \{M_{h_2} \}_{H^2_{\beta}}' = \{M_z \}_{H^2_{\beta}}'$, and let $m_1$ and $m_2$ be two positive integers. Then
\[
\bigoplus\limits_1^{m_1} M_{h_1}\sim\bigoplus\limits_1^{m_2} M_{h_2}
\]
if and only if $m_1=m_2$ and $M_{h_1}\sim M_{h_2}$, i.e., there exists a M\"{o}bius transformation $\varphi\in \textrm{Aut}(\mathbb{D})$ such that $h_2=h_1\circ \varphi$.
\end{theorem}

\begin{proof}
Suppose that
\[
\bigoplus\limits_1^{m_1} M_{h_1}\sim\bigoplus\limits_1^{m_2} M_{h_2}.
\]
Let $T=M_{h_1}^{(m_1)}\oplus M_{h_2}^{(m_2)}$. Then
\[
T\sim M_{h_1}^{(2m_1)}.
\]
Since,
\[
\{M_{h_1}^{(2m_1)} \}_{H^2_{\beta}}' =\{M_F; F\in {\textrm M_{2m_1}}(H^{\infty}_{\beta})\},
\]
it follows from Lemma 2.9 in \cite{CFJ} or Theorem 6.11 in \cite{JW} that
\[
\bigvee(\{T \}_{H^2_{\beta}}')\cong \bigvee(\{M_{h_1}^{(2m_1)} \}_{H^2_{\beta}}') \cong \bigvee({\textrm M_{2m_1}}(H^{\infty}_{\beta}))\cong \bigvee (H^{\infty}_{\beta}) \cong\mathbb{N}.
\]
In addition, $\{M_{h_1} \}_{H^2_{\beta}}' = \{M_{h_2} \}_{H^2_{\beta}}' = \{M_z \}_{H^2_{\beta}}'=H^{\infty}_{\beta}$ implies $M_{h_1}$ and $M_{h_2}$ are strongly irreducible.
Thus, by Corollary \ref{n1}, we have $M_{h_1}\sim M_{h_2}$.

On the other hand, there exists $\lambda\in\mathbb{C}$ such that $M_{h_1-\lambda}$ is Fredholm with nonzero index. Following from $M_{h_1}\sim M_{h_2}$, $M_{h_2-\lambda}$ is also Fredholm with the same index of $M_{h_1-\lambda}$.
Since $M_{h_1-\lambda}^{(m_1)}\sim M_{h_2-\lambda}^{(m_2)}$, we have
\[
m_1\cdot {\textrm ind}M_{h_1-\lambda}={\textrm ind}M_{h_1-\lambda}^{(m_1)}={\textrm ind}M_{h_2-\lambda}^{(m_2)}=m_2\cdot {\textrm ind}M_{h_2-\lambda}.
\]
Then $m_1=m_2$.

The converse is obvious. The proof is finished.
\end{proof}

{\bf Proof of Theorem \ref{JPT}.}
Following from Theorem \ref{JDT} (the existence) and Theorem \ref{nh} (the uniqueness in the sense of analytic autoumorphism group actiion), we obtain the Jordan representation theorem of $\textrm{Hol}(\overline{\mathbb{D}})$ on a weighted Hardy space $H^2_{\beta}$ of polynomial growth immediately. \qed

As a corollary, we may characterize the similarity classification of the representation of $\textrm{Hol}(\overline{\mathbb{D}})$ on a weighted Hardy space $H^2_{\beta}$ of polynomial growth, which generalizes the main result of Jiang and Zheng in \cite{JZ}.

{\bf Proof of Theorem \ref{Simlar}.}
Suppose that $M_{f_1}$ is similar to $M_{f_2}$ on $H^2_{\beta}$. By Theorem \ref{fhB}, we may write
\[
f_1=h\circ B_1, \ \ \ \ f_2=h_2\circ \widetilde{B_2},
\]
where $h, h_2 \in \textrm{Hol}(\overline{\mathbb{D}})$ such that $M_{h}$ and $M_{h_2}$ are strongly irreducible, and $B_1$ and
$\widetilde{B_2}$ are two finite Blaschke products with order $m$ and $m_2$, respectively. Following from Theorem \ref{JDT} and Theorem \ref{nh}, we have
$m=m_2$ and $M_{h}\sim M_{h_2}$, i.e., there exists a $\varphi\in \textrm{Aut}(\mathbb{D})$ such that $h_2=h\circ \varphi$. Let $B_2=\varphi\circ \widetilde{B_2}$. Then $B_2$ is also a Blaschke product with order $m$ and
\[
f_2=h_2\circ \widetilde{B_2}=h\circ \varphi\circ \widetilde{B_2}= h\circ B_2.
\]

The converse is a straightforward corollary of Theorem \ref{JDT}. The proof is finished.  \qed

\noindent {\Large \textbf{Declarations}}

\noindent \textbf{Ethics approval}

\noindent Not applicable.

\noindent \textbf{Competing interests}

\noindent The author declares that there is no conflict of interest or competing interest.

\noindent \textbf{Acknowledgement}

\noindent The authors thank Professor Lixin Cheng for many valuable discussions on this paper.

\noindent \textbf{Funding}

\noindent The second author was supported by National Natural Science Foundation of China (Grant No. 11831006, 11920101001 and 11771117).

\noindent \textbf{Availability of data and materials}

\noindent Data sharing is not applicable to this article as no datasets were generated or analyzed during the current study.

\end{document}